\documentclass[a4paper]{amsart}
\usepackage{verbatim}
\usepackage[dvips]{color}
\usepackage{amssymb}

\author[F. Duzaar]{Frank Duzaar}
\address{Frank Duzaar\\Department Mathematik, Universit\"at
Erlangen--N\"urnberg\\ Bismarckstrasse 1 1/2, 91054 Erlangen, Germany}
\email{duzaar@mi.uni-erlangen.de}

\author[G. Mingione]{Giuseppe Mingione}
\address{Dipartimento di Matematica, Universit\`a di Parma\\
Viale G.~P.~Usberti 53/a, Campus, 43100 Parma, Italy}
\email{giuseppe.mingione@unipr.it.}

\allowdisplaybreaks

\newcommand{\oscs}{\mathop{\mathrm{osc}}}

\newtheorem{theorem}{Theorem}[section]
\newtheorem{prop}{Proposition}[section]
\newtheorem{lemma}{Lemma}[section]
\newtheorem{cor}{Corollary}[section]
\theoremstyle{definition}

\newtheorem{remark}{Remark}[section]

\numberwithin{equation}{section}

\newcommand{\tta}{\tilde{a}}

\newcommand{\rif}[1]{(\ref{#1})}

\newcommand\ap{``}

\newcommand{\tu}{\tilde{u}}

\def\eqn#1$$#2$${\begin{equation}\label#1#2\end{equation}}
\def\charfn_#1{{\raise1.2pt\hbox{$\chi
_{\kern-1pt\lower3pt\hbox{{$\scriptstyle#1$}}}$}}}

\def\qq1{q_*}
\def\q2{q_{**}}

\def\ep{\varepsilon}
\def\en{\mathbb N}
\def\er{\mathbb R}

\def\loc{\operatorname{loc}}

\newdimen\vintbar
\def\vint{-\kern-\vintbar\int}

\def\H{\mathcal H}

\def\0{\boldsymbol 0}

\newcommand{\ratio}{\nu, L}

\newcommand{\MM}{\mathcal{M}}

\newcommand{\divo}{\textnormal{div}}

 \newcommand{\mean}[1]{-\hskip-1.08em\int_{#1}}

\newcommand{\npma}{{\bf W}_{\frac{1}{p},p}^{\mu}}

\newcommand{\ww}{{\bf W}_{1,p}^{\mu}}

\newcommand{\trif}[1] {\textnormal{\rif{#1}}}

\newtoks\by
\newtoks\paper
\newtoks\book
\newtoks\jour
\newtoks\yr
\newtoks\pages
\newtoks\vol
\newtoks\publ
\def\et{ \& }
\def\name[#1, #2]{#1 #2}
\def\ota{{\hbox{\bf ???}}}
\def\cLear{\by=\ota\paper=\ota\book=\ota\jour=\ota\yr=\ota
\pages=\ota\vol=\ota\publ=\ota}
\def\endpaper{\the\by, \textit{\the\paper},
{\the\jour} \textbf{\the\vol} (\the\yr), \the\pages.\cLear}
\def\endbook{\the\by, \textit{\the\book},
\the\publ, \the\yr.\cLear}
\def\endpap{\the\by, \textit{\the\paper}, \the\jour.\cLear}
\def\endproc{\the\by, \textit{\the\paper}, \the\book, \the\publ,
\the\yr, \the\pages.\cLear}

\title[Gradient estimates via non-linear potentials] {Gradient estimates via non-linear potentials}
\begin{document}
to appear in the American Journal of Mathematics

\maketitle

\begin{abstract}
We present pointwise gradient bounds for solutions to $p$-Laplacean type non-homogeneous equations employing non-linear Wolff type potentials, and then prove similar bounds, via suitable caloric potentials, for solutions to parabolic equations. 
\end{abstract}
\tableofcontents
\section{Introduction and main results}
The aim of this paper is to give a somehow unexpected but nevertheless natural maximal order - and parabolic - version of a by now classical result due to Kilpel\"ainen \& Mal\'y \cite{KM}, later extended, by mean of a different approach, by Trudinger \& Wang \cite{TW}. In \cite{KM, TW} the authors obtained a neat pointwise bound - see \rif{KM} below - for solutions to non-homogeneous quasi-linear equations of $p$-Laplacean type by mean of certain natural Wolff type non-linear potentials of the right-hand side measure datum; this is commonly considered as a basic result in the theory of quasi-linear equation. In this paper {\em we upgrade such a result by showing that similar pointwise bounds also hold for the gradient of solutions}, see \rif{mainestx} below. The resulting estimate finally allows to draw a complete and unified picture concerning gradient estimates for degenerate problems: it permits to recast classical gradient estimates for the non-homogeneous $p$-Laplacean equation obtained by Iwaniec \cite{I} and DiBenedetto \& Manfredi \cite{DM}, as well as the pointwise $L^\infty$ estimates obtainable when the right hand side is integrable enough \cite{Ma, Manth}, up to the classical gradient bounds of Boccardo \& Gallou\"et \cite{BG1, BG2}, Lindqvist \cite{Lind} and Dolzmann \& Hungerb\"uhler \& M\"uller \cite{DHM} for measure data problems; for more on the $p$-Laplacean see also the notes of Lindqvist \cite{Lnotes}. More significantly, the results here allow for {\em a characterization of Lipschitz continuity of solutions in terms of Riesz potentials which is analogous to the one available for the standard Poisson equation}; see \rif{minass} and \rif{mainestxr} below. Furthermore, the estimates extend to the case of coefficient dependent equations, allowing to get regularity criteria 
with respect to the regularity of the coefficients which are exactly those ones suggested by the analysis of the fundamental solutions of linear equations \cite{GW}. Our elliptic results find a natural but non-trivial analogue in the parabolic case, as we are also going to show here for a significant class of parabolic problems. In this case we shall employ a natural family of \ap caloric" type Riesz potentials to give a pointwise bound for the spatial derivative of solutions; in turn, the resulting inequality implies relevant sharp borderline integral estimates which, known in the elliptic case, are otherwise unreachable via the known techniques, in the parabolic one. 
In particular, this finally allows to make the missing link between elliptic and parabolic borderline estimates such as those in Lorentz spaces. Moreover, as in the elliptic case, {\em we find a sharp characterization of the local boundedness of the gradient via
Caloric potentials}; see \rif{minasspar} below. It is worth mentioning that the Lipschitz continuity estimates and the techniques of this paper are the starting point for further developments, eventually leading to $C^{1}$-regularity assertions for solutions \cite{DMcont}.
For notation and more details we recommend the reader to look at Section 2 below.

\subsection{Degenerate Elliptic estimates}
In this section the growth exponent $p$ will be a number such that $p\geq 2$, while for the rest of the paper $\Omega$ will denote a bounded open and Lipschitz domain of $\er^n$, with $n\geq 2$.
We shall consider general non-linear, possibly degenerate elliptic equations with $p$-growth of the type
\eqn{baseq}
$$
-\divo \ a(x,Du)=\mu\,.
$$
whenever $\mu$ is a Radon measure defined on $\Omega$ with finite total mass; eventually letting $\mu(\er^n \setminus \Omega)=0$, without loss of generality we may assume that $\mu$ is defined on the whole $\er^n$. 
The continuous vector field $a \colon \Omega \times \er^n \to \er^n$ is assumed to be $C^1$-regular in the gradient variable $z$, with $a_z(\cdot)$ being Carath\'eodory regular, and
satisfying the following {\em growth, ellipticity and continuity assumptions}:
\eqn{asp}
$$
\left\{
    \begin{array}{c}
    |a(x,z)|+|a_{z}(x,z)|(|z|^2+s^2)^{\frac{1}{2}} \leq L(|z|^2+s^2)^{\frac{p-1}{2}} \\ [3 pt]
    \nu^{-1}(|z|^2+s^2)^{\frac{p-2}{2}}|\lambda|^{2} \leq \langle a_{z}(x,z)\lambda, \lambda
    \rangle\\[3pt]
    |a(x,z)-a(x_0,z)|\leq L_1\omega(|x-x_0|)(|z|^2+s^2)^{\frac{p-1}{2}}
    \end{array}
    \right.
$$
whenever $x,x_0 \in \Omega$ and $z, \lambda \in \er^n$, where
$0< \nu\leq L$ and $s\geq 0, L_1\geq 1$ are fixed parameters. In \rif{asp} the function $\omega\colon [0,\infty) \to [0,\infty)$ is a modulus of continuity i.e.,
a non-decreasing concave function such that $\omega(0)=0=\lim_{\varrho\downarrow 0}\omega (\varrho)$ and $\omega(\cdot)\leq 1$.
On such a function we impose a natural decay property, which is essentially optimal for the result we are going to have, and prescribes a {\em Dini-continuous dependence of the partial map} $x \mapsto a(x,z)/(|z|+s)^{p-1}$:
\eqn{intdini}
$$
\int_0^R [\omega(\varrho)]^{\frac{2}{p}}\,\frac{ d \varrho}{\varrho} := d(R)< \infty\,.
$$
The prototype of \rif{baseq} is - choosing $s=0$ and omitting the $x$-dependence - clearly given by the $p$-Laplacean equation
\eqn{plap}
$$
-\divo\, (|Du|^{p-2}Du)=\mu\,.
$$
Our estimate involves the classical non-linear Wolff potential defined by
\eqn{newpot}
$$
{\bf W}^{\mu}_{\beta, p}(x_0,R):= \int_0^R \left(\frac{|\mu|(B(x_0,\varrho))}{\varrho^{n-\beta p}}\right)^{\frac{1}{p-1}}\, \frac{d\varrho}{\varrho}\qquad \qquad  \beta \in (0,n/p]\,.
$$
The first result, which we naturally state as an a priori estimate, is
\begin{theorem}[Non-linear potential gradient bound]\label{mainx} Let $u \in C^{1}(\Omega)$ be a weak solution to \trif{baseq}
 with $\mu \in L^1(\Omega)$ under the assumptions \trif{asp}. Then there exists a constant $c \equiv c (n,p,\ratio)>0$, and a positive radius $\tilde{R}\equiv \tilde{R}(n,p,\ratio,L_1, \omega(\cdot))$, such that
the pointwise estimate
\eqn{mainestx}
$$
|Du(x_0)| \leq c\mean{B(x_0,R)}(|Du|+s)\, dx  + c\npma(x_0,2R)
$$
holds whenever $B(x_0,2R)\subseteq \Omega$ and $R\leq \tilde{R}$. Moreover, when the vector field $a(\cdot)$ is independent of $x$ - and in particular for the $p$-Laplacean operator \trif{plap} - estimate \trif{mainestx} holds without any restriction on $R$.
\end{theorem}
The potential $\npma$ appearing in \rif{mainestx} is the natural one since its shape respects the scaling properties of the equation with respect to the estimate in question - see \underline{}Section \ref{ssee} and compare with the linear estimate in \rif{stima0}; moreover, the potential $\npma$ is in a sense {\em optimal}, as in an estimate like \rif{mainestx} it implies the recovering of the sharp integrability results known for general weak solutions to \trif{baseq}. Moreover, as a matter of fact when the right hand side is a properly singular measure estimate \rif{mainestx} reverses; for all such aspects we refer to Remark \ref{wsharp} below. An approximation procedure eventually allows to remove the extra regularity assumptions $u \in C^{1}(\Omega)$ and $\mu \in L^1(\Omega)$
up to consider the most general case of solutions to measure data problems as
\eqn{Dir1}
$$
\left\{
    \begin{array}{cc}
    -\divo \ a(x,Du)=\mu & \qquad \mbox{in $\Omega$}\\
        u= 0&\qquad \mbox{on $\partial\Omega$\,.}
\end{array}\right.
$$ Our result indeed holds for general Solutions Obtained by Limit of Approximations (\textnormal{SOLA}) to \rif{Dir1}, a class of very weak solutions - solutions not necessarily lying in the natural energy space $W^{1,p}(\Omega)$ - which is introduced to get existence and uniqueness in several cases. Moreover, SOLA coincide with usual weak solutions for regular data $\mu \in W^{-1,p'}(\Omega)$; we refer to Section \ref{app} below for more details, see also Theorems \ref{mainx2}-\ref{mainx22} below. When extended to general weak solutions {\em estimate \trif{mainestx} tells us the remarkable fact that the boundedness of $Du$ at a point $x_0$ is independent of the solution $u$, and of the vector field $a(\cdot)$, but only depends on the behavior of $|\mu|$ in a neighborhood of $x_0$.}
\begin{cor}[$C^{0,1}$-regularity criterium]\label{corollario} Let $u \in W^{1,p-1}_0(\Omega)$ be a \textnormal{SOLA} to the problem \trif{Dir1} - which is unique in the case $\mu \in L^1(\Omega)$ -
under the assumptions \trif{asp} and \trif{intdini}. Then
\eqn{minass}
$$\npma(\cdot,R) \in L^{\infty}(\Omega) \ \mbox{for some}\  R>0 \Longrightarrow Du \in L^{\infty}_{\loc}(\Omega, \er^n)\,.$$
Furthermore, there exists a constant $c$, depending only on $n,\ratio, L_1, d(\cdot)$, such that the following estimate holds whenever $B_{2R}\subseteq \Omega$:
\eqn{minimalcri}
$$
\|Du\|_{L^{\infty}(B_{R/2})} \leq c \mean{B(x_0,R)}(|Du|+s)\, dx+ c\left\|\npma(\cdot,R)\right\|_{L^{\infty}(B_{R})} \,.
$$
\end{cor}
\begin{remark} Corollary \ref{corollario} allows for a natural Lipschitz continuity criterium with respect to the regularity of coefficients \rif{intdini} and moreover to obtain for the $p$-Laplacean, essentially the same criterium of Lipschitz continuity available for the Laplacean, that is the one via Riesz potentials; see \rif{mainestxr} below and also \rif{minass2} for a corollary. {\em We remark that finding conditions on $\mu$ implying the boundedness of $Du$ was a major open issue in the theory of $p$-Laplacean equation and \trif{minass} seems to provide a first satisfying answer to the problem.} Indeed Cianchi \cite{Ci0} has shown that already in the plain case of the standard Poisson equation \eqn{poisson}
$$
-\triangle u = \mu
$$ the gradient might be unbounded as soon as ${\bf W}_{\frac{1}{2},2} \not \in L^\infty$.
The estimates above also emphasize the role played by Dini-continuity of the coefficients in order to get
the boundedness of the gradient, which is basically the same one observed when dealing with linear elliptic equations with variable Dini-continuous coefficients, and proving pointwise bounds on the derivatives of the Green's function \cite[Section 3]{GW}. As a matter of fact a feature of the approach adopted here is that estimate \rif{mainestx} allows to derive conditions for gradient boundedness, {\em which are borderline simultaneously} w.r.t.~the right hand side - i.e.~\rif{minass} - and w.r.t.~coefficients - i.e.~\rif{intdini}, already in the linear case. Indeed, a recent striking example of Jin \et Mazya \et Van Schaftingen \cite{mazyaex} shows that there exist distributional solutions $u$ to linear equations of the type
$$
\divo\,  (A(x)Du)=0
$$
with a continuous (not Dini-continuous) and elliptic matrix $A(\cdot)$, such that $u \in W^{1,p}$ for every $p<\infty$ as predicted by the standard regularity theory, but such that $u \not\in W^{1,\infty}$; in this case one even has that $Du \not\in$ BMO.
\end{remark}
\subsection{Zero-order and related estimates.}\label{ssee}
The linkage between estimate \rif{mainestx} and the estimates of \cite{KM,TW} is clear. The linear case
$
a(x,z)=z
$
obviously leads to the standard Poisson equation
\rif{poisson}
- here for simplicity considered in the whole $€\er^n$ - for which, due to the use of classical representation formulas, it is possible to get pointwise bounds for solutions via the use of Riezs potentials
\eqn{ellr}
$$
I_\beta(\mu)(x):=\int_{\er^n}\frac{d \mu(y)}{|x-y|^{n-\beta}}\,,\qquad \qquad \beta \in (0,n]
$$
such as
\eqn{stima0}
$$
|u(x)|\leq  c I_2(|\mu|)(x) \qquad \mbox{and}\qquad |Du(x)|\leq  c I_1(|\mu|)(x) \,,
$$
the first being actually valid for $n\geq 3$. We recall that the equivalent, localized version of the Riesz potential $I_\beta(\mu)$
is given by the linear potential
\eqn{WWRR}
$$
{\bf I}_{\beta}^\mu(x_0,R)\equiv {\bf W}_{\frac{\beta}{2},2}^\mu(x_0,R)= \int_0^R \frac{\mu(B(x_0,\varrho))}{\varrho^{n-\beta}}\, \frac{d\varrho}{\varrho}\leq I_\beta(\mu)(x_0)\,,
$$
the last line being valid for non-negative measures (note that we use two different - but similar - notations for the Riesz potential in \rif{ellr} and its \ap polar version" in \rif{WWRR}). The natural non-linear version of \rif{stima0}, derived in \cite{KM,TW} {\em for non-negative measures}, when $p\leq n$, is treated via Wolff potentials:
\eqn{KM}
$$
|u(x_0)| \leq c  \left(\mean{B(x_0,R)}(|u|+s)^{\gamma}\, dx\right)^{\frac{1}{\gamma}}  + c\ww(x_0,2R)\qquad \gamma >p-1\,.
$$
Estimate \rif{mainestx} {\em upgrades} \rif{KM} to the gradient/maximal level, obviously replacing $\ww$ with ${\bf W}^{\mu}_{\frac{1}{p},p}$, and
represents the $p$-Laplacean analog of the second estimate in the left hand side of \rif{stima0}. Indeed, for $p=2$ by \rif{WWRR} we have ${\bf W}_{\frac{1}{2},2}^\mu(x_0,R)\leq I_1(|\mu|)(x_0)$.
 Let us mention that the technique developed for estimate \rif{mainestx} also yields an alternative proof of estimate \rif{KM} which now holds {\em for general signed measures}; see Remark \ref{alternative} below.
 We also refer to \cite{mis3} for gradient potential estimates when $p=2$, and to the important work of Labutin \cite{labutin} for relevant Wolff type potential estimates related to fully non-linear Hessian type operators. For further interesting relations between degenerate quasilinear equations and potentials we refer to the recent interesting paper of Lindqvist \& Manfredi \cite{LiMa}. Another consequence of estimate \rif{mainestx} and of the classical bound
\begin{eqnarray}
\nonumber \npma(x_0,\infty) & = & \int_{0}^\infty \left(\frac{|\mu|(B(x_0,\varrho))}{\varrho^{n-1}}\right)^{\frac{1}{p-1}} \frac{d \varrho}{\varrho}\\ & \leq &
 c I_{\frac{1}{p}}\left\{\left[I_{\frac{1}{p}}(|\mu|)\right]^{\frac{1}{p-1}}\right\}(x_0)=: c {\bf V}_{\frac{1}{p},p}(|\mu|)(x_0)\label{rieszbound}
\end{eqnarray}
is the estimate
\eqn{mainestxr}
$$
|Du(x_0)| \leq c\mean{B(x_0,R)}(|Du|+s)\, dx  + c {\bf V}_{\frac{1}{p},p}(|\mu|)(x_0)
$$
which holds whenever $B(x_0,2R)\subseteq \Omega$ satisfies the smallness condition imposed in Theorem \ref{mainx}. Here we remind the reader that we have previously extended $\mu$ to the whole space $\er^n$. The non-linear potential ${\bf V}_{\frac{1}{p},p}(\mu)(x_0)$ - often called the Havin-Maz'ja potential of $\mu$ - is a classical object in non-linear potential theory, and together with the bound \rif{rieszbound} comes from the fundamental and pioneering work of Adams \& Meyers and Havin \& Maz'ja; see also \cite{AdHe, AdMe, MH}. Estimate \rif{mainestxr} allows to derive all types of local estimates starting by the properties of the Riesz potential; see Section \ref{localest} below.
\subsection{Parabolic estimates} Our aim here is not only to give a parabolic version of the elliptic estimate \rif{mainestx}, but also to give a zero order estimate, that is the parabolic analog of the zero order elliptic estimate \cite{KM}, which at this point will essentially follow as a corollary of the proof of the gradient estimate. 
We consider quasi-linear parabolic equations of the type
\eqn{basicpar}
$$
u_t - \divo \ a(x,t,Du)= \mu\,,
$$
in a cylindrical domain $\Omega_T:= \Omega\times (-T,0)$, where as in the previous sections $\Omega
\subseteq \er^n$, $n \geq 2$ and $T>0$.
The vector-field $a\colon\Omega_T\times\er^n\to\er^n$ is assumed to be
Carath\`eodory regular together with $a_z(\cdot)$, and indeed being $C^1$-regular with respect to the gradient variable $z \in \er^n$, and satisfying the following standard growth, ellipticity/parabolicity and continuity conditions:
\begin{equation}\label{par1}
    \left\{
    \begin{array}{c}
 |a(x,t,z)|+|a_{z}(x,t,z)|(|z|+s) \leq L(|z|+s) \\[4pt]
    \nu|\lambda|^2 \leq \langle a_{z}(x,t,z)\lambda, \lambda\rangle  \\[4pt]
    |a(x,t,z)-a(x_0,t,z)|\leq L_1\omega(|x-x_0|)(|z|+s)
    \end{array}
    \right.
\end{equation}
for every choice of $x, x_0\in\Omega$, $z,\lambda \in\er^n$ and $ t \in (-T,0)$; here the function $\omega\colon [0, \infty) \to [0,1]$ is as in \rif{asp}$_3$ for $p=2$. Note that anyway we are assuming no continuity on the map $t \mapsto a(\cdot, t, \cdot)$, which is considered to be a priori only measurable. In other words we are considering the analog of assumptions \rif{asp} for $p=2$; the reason we are adopting this restriction is that when dealing with the evolutionary $p$-Laplacean operator estimates take the usual form only when using so called \ap intrinsic cylinders", according the by now classical parabolic $p$-Laplacean theory developed by DiBenedetto \cite{D}. These are - unless $p=2$ when they reduce to the standard parabolic ones - cylinders whose size locally depends on the size of the solutions itself, therefore a formulation of the estimates via non-linear potentials - whose definition is built essentially using a standard family of balls and it is therefore \ap universal" - is not clear and will be the object of future investigation. We refer to \cite{AM} for global gradient estimates.  

In order to state our results we need some additional terminology. Let us recall that given points $(x,t), (x_0,t_0)\in\er^{n+1}$ the standard parabolic metric is defined by
\eqn{parme}
$$
d_{\rm par}((x,t),(x_0,t_0)):=\max\{|x-x_0|, \sqrt{|t-t_0|}\} \thickapprox \sqrt{|x-x_0|^2+|t-t_0|}
$$
and the related metric balls with radius $R$ with respect to this metric are given by the cylinders of the type $B(x_0,R)\times (t_0-R^2, t_0+R^2)$. The \ap caloric" Riesz potential - compare with elliptic one defined in \rif{ellr} and see Remark \ref{parpot} below - is now built starting from
\begin{equation}\label{parr}
    I_\beta (\mu)((x,t)):=\int_{\er^{n+1}}\frac{d\mu((\tilde x,\tilde t ))}{d_{\rm par}((\tilde x,\tilde t ),(x,t))^{N-\beta}}
    \qquad \quad 0 < \beta \leq N:=n+2
\end{equation}
whenever $(x,t) \in \er^{n+1}$. In order to be used in estimates for parabolic equations, it is convenient to introduce its local version via the usual backward parabolic cylinders - with \ap vertex" at $(x_0,t_0)$ - that is
\eqn{caloricc}
$$
Q(x_0,t_0;R):= B(x_0,R)\times (t_0-R^2, t_0)\,,
$$
and is now given by - with $N := n+2$
\eqn{calpot}
$$
 {\bf I}_{\beta}^\mu(x_0,t_0;R):=
 \int_0^R \frac{|\mu|(Q(x_0,t_0;\varrho))}{\varrho^{N-\beta}}\, \frac{d\varrho}{\varrho}\qquad \mbox{where}\ \beta \in (0,N]\,.
$$
The main result in the parabolic case is
\begin{theorem}[Caloric potential gradient bound]\label{mainp} Under the assumptions \trif{par1} and \trif{intdini}, let $u\in C^0(-T,0; L^2(\Omega))$ be a weak solution to \trif{basicpar} with $\mu \in L^2(\Omega_T)$, and such that $Du\in C^0(\Omega_T)$. Then there exists a constant $c\equiv c(n,\nu, L) $ and a radius $\tilde{R}\equiv \tilde{R}(n,\nu, L, L_1, \omega(\cdot)) $ such that the following estimate:
\eqn{parest1}
$$
|D u(x_0,t_0)|  \leq c \mean{Q(x_0,t_0;R)}(|D u|+s)\, dx\, dt + c{\bf I}_{1}^\mu(x_0,t_0;2R)
$$
holds whenever $Q(x_0,t_0; 2R)\subseteq \Omega$ is a backward parabolic cylinder with vertex at $(x_0,t_0)$ and such that $R \leq \tilde{R}$.
\end{theorem}
As in the elliptic case when dealing with parabolic equations with no space variable dependence of the type
 \eqn{basicparaa}
 $$
u_t - \divo\, a(t,Du)=\mu
$$
 we can improve the result of Theorem \ref{mainp} as follows:
\begin{theorem}[Componentwise caloric bound]\label{mainpa} Under the assumptions \trif{par1}, let $u\in C^0(-T,0; L^2(\Omega))$ be a weak solution to \trif{basicparaa} with $\mu \in L^2(\Omega_T)$, and such that $Du\in C^0(\Omega_T)$. Then there exists a constant $c\equiv c(n,\nu, L) $ such that the following estimate:
\eqn{parest1cc}
$$
|D_\xi u(x_0,t_0)|  \leq c \mean{Q(x_0,t_0;R)}(|D_\xi u|+s)\, dx\, dt + c{\bf I}_{1}^\mu(x_0,t_0;2R)
$$
holds whenever $Q(x_0,t_0; 2R)\subseteq \Omega$ is a backward parabolic cylinder with vertex at $(x_0,t_0)$, and $\xi \in \{1,\ldots,n\}$.
\end{theorem}
Estimates \rif{parest1}-\rif{parest1cc} also hold for general weak and very weak solutions, and in particular for solutions to measure data problems as
\begin{equation}\label{parpro}
    \left\{
    \begin{array}{cc}
    u_t-\divo \, a(x,t,Du)=\mu&\mbox{in $ \Omega_T$}\\[3pt]
    u=0&\mbox{on $\partial_{\rm par}\Omega_T$\,,}
    \end{array}
    \right.
\end{equation}
where $\mu$ is a general Radon measure on $\Omega_T$ with finite mass on $\Omega_T$, that we shall consider to be defined in the whole $\er^{n+1}$. See next section for the definition of solutions. In the spirit of the elliptic result \rif{minimalcri} we have the following implication, which provides a boundedness criteria for the spatial gradient, under the Dini-continuity assumption for the spatial coefficients \rif{intdini}:
\eqn{minasspar}
$${\bf I}_{1}^\mu(\cdot;R) \in L^{\infty}(\Omega_T), \ \mbox{for some}\  R>0 \Longrightarrow Du \in L^{\infty}_{\loc}(\Omega_T, \er^n)\,.$$
We conclude with the zero order potential estimate, which applies to general equations of the type \rif{basicpar} when considered with a measurable dependence upon the coefficients $(x,t)$. The relevant hypotheses in this case are the following standard growth and monotonicity properties:
\begin{equation}\label{par2}
    \left\{
    \begin{array}{c}
 |a(x,t,z)|\leq L(|z|+s) \\[4pt]
    \nu|z_2-z_1|^2 \leq \langle a(x,t,z_2)-a(x,t,z_1), z_2-z_1\rangle
    \end{array}
    \right.
\end{equation}
which are assumed to hold whenever $(x,t)\in \Omega_T$ and $z,z_1,z_2 \in \er^n$. In particular, since the pointwise bound will be derived on $u$, rather than on $Du$, we do not need differentiability assumptions for $a(\cdot)$ with respect to the spatial gradient variable $z$.
\begin{theorem}[Zero order estimate]\label{mainp2} Under the assumptions \trif{par2}, let $$u \in  L^2(-T, 0;W^{1,2}(\Omega))\cap C^0(\Omega_T)$$ be a weak solution to \trif{basicpar} with $\mu \in L^2(\Omega_T)$. Then there exists a constant $c$, depending only on $n,\nu, L, L_1$, such that the following inequality holds whenever $Q(x_0,t_0; 2R)\subseteq \Omega$ is a backward parabolic cylinder with vertex at $(x_0,t_0)$:
\eqn{lastpest}
$$
|u(x_0,t_0)| \leq c \mean{Q(x_0,t_0;R)}(|u|+Rs)\, dx\, dt  +  c{\bf I}_{2}^\mu(x_0,t_0;2R)  \,.
$$
\end{theorem}
\begin{remark}\label{parpot} The caloric Riesz potential defined in \rif{parr} is different from the parabolic Riesz potential considered for instance  in \cite{AdBa} via convolution of $\mu$ with the heat kernel
\eqn{kuusi}
$$
\Gamma_{\beta}(x,t):= \frac{1}{t^{\frac{N-\beta}{2}}}\exp\left(-\frac{|x|^2}{4t}\right) \qquad \qquad \beta \in (0,N]\,.
$$
On the other hand both the kernels considered exhibit the same scaling with respect to the parabolic dilation $R \to (Rx, R^2t)$, which is in turn the relevant property to determine the regularization properties of the related convolution operator. Therefore, as we shall see from Section \ref{parri} below, we have that the caloric Riesz potential \rif{parr} is perfectly sufficient to infer the sharp regularity/integrability properties of solutions. In particular, from \rif{parest1} we shall be able to derive a borderline Marcinkievicz estimate which seems to be difficult to derive using for instance the truncation methods from \cite{Bmany}. Anyway, for estimates regarding \rif{kuusi} see \cite{kuusim}.
\end{remark}
\subsection{Plan of the paper} After establishing in Section 2 some notation, Section \ref{ellipticse} is dedicated to the proof of Theorem \ref{mainx}. This will require a careful combination of regularity estimates for $p$-harmonic functions, together with suitable comparison estimates which make the density of the Wolff potential appear; all such ingredients will be finally combined in a suitable iteration scheme. The same scheme will be followed in Section \ref{paraproof} for the parabolic case, where additional difficulties come into the play; in particular, the proof of the comparison estimates necessitates further delicate technicalities, while some precise estimates from Nash-Moser's theory will be needed. In Section 5 we show how the everywhere pointwise estimates derived for a priori regular solutions, actually extend to all kinds of general weak solutions, including solutions to measure data problems, which are the most general ones. Finally, in Section 6 we demonstrate how the pointwise estimates derived here allows to recast essentially all the main gradient estimates for non-homogeneous problems as described at the beginning of this Introduction, and in particular those in finer scales of spaces such as Lorentz or Orlicz spaces. We re-emphasize here that, when considering parabolic problems, the techniques presented here are the only one available for getting such estimates. The techniques presented in this paper are general enough to be applied in different contexts; an example is \cite{BH}, where problems with non-standard growth of $p(x)$-type - see for instance \cite{AMpx} - are considered. 

Some of the results here have been announced in the {\em nota lincea} \cite{DMnote}.

{\bf Acknowledgments.} This research is supported by the ERC grant 207573 ``Vectorial Problems", and by MIUR via the national project \ap Calcolo delle Variazioni". We also thank Verena B\"ogelein who carefully read a draft of the paper.
\section{Notations}
In this paper we follow the usual convention of denoting by $c$ a general constant larger (or equal) than one, possibly varying from line to line; special occurrences will be denoted by $c_1$ etc; relevant dependence on parameters will be emphasized using parentheses. We shall denote in a standard way $$B(x_0,R):=\{x \in \er^n \, : \,  |x-x_0|< R\}$$ the open ball with center $x_0$ and radius $R>0$, while backward parabolic cylinders have been already defined in \rif{caloricc}; when not important, or clear from the context, we shall omit denoting the center as follows: $B_R \equiv B(x_0,R)$, and the same will happen for parabolic cylinders concerning the vertex. Moreover, when more than one ball - resp. cylinder - will come into the play, they will always share the same center - resp. vertex - unless otherwise stated. We shall also denote $B \equiv B_1 = B(0,1)$, and $Q\equiv Q_1 = B_1 \times (-1,0)$. We recall that, given a cylindrical domain of the type $C=A \times (t_1, t_2)$ - and in particular a parabolic cylinder - its parabolic boundary $\partial_{\rm par} C$ is given by $\partial C \setminus (A \times \{t_2\})$. With $A$ being a measurable subset with positive measure, and with $g \colon A \to \er^k$ being a measurable map, we shall denote
 $$
    \mean{A}  g(x) \, dx  := \frac{1}{|A|}\int_{A}  g(x) \, dx
$$
its average. In this paper, all the measures considered will be Radon measure with finite total mass, for convenience assumed to be defined on the whole space: $\er^n$ for elliptic problems, $\er^{n+1}$ for parabolic ones; with $A\subset \er^n$ or $A\subset \er^{n+1}$, we shall denote by $\MM(A)$ the space of all Radon measures with finite total mass defined on $A$.
When a measure $\mu$ actually turns out to be an $L^1$-function, we shall use the standard notation, whenever $A$ is a measurable set on which $\mu$ is defined
$$
|\mu|(A):=\int_{A} |\mu(x)|\, dx\,.
$$
When dealing with general measure data problems as in \rif{baseq} and \rif{basicpar}, we shall consider so called very weak solutions, i.e. those solutions who are not necessarily lying in the natural energy spaces associated to such equations growth assumptions \rif{asp}$_1$ and \rif{par1}$_1$ being given - specifically $W^{1,p}(\Omega)$ and $L^2(-T,0;W^{1,2}(\Omega))$, respectively - but that are nevertheless integrable enough to allow for a  usual distributional formulation. Such solutions are not necessarily unique, but suitable reinforced notions of solutions can be considered as well, in order to achieve uniqueness in several cases; see Section 5 below. In the elliptic case a (very) weak solution to \rif{baseq} is a function $u \in W^{1,p-1}(\Omega)$ such that the distributional relation
$$
\int_\Omega \langle a(x,Du), D\varphi\rangle\, dx =  \int_\Omega \varphi\, d \mu
$$
holds whenever $\varphi \in C^\infty_0(\Omega)$ has compact support in $\Omega$. In the parabolic case a (very) weak solution is a function $u \in L^1(-T, 0; W^{1,1}(\Omega))$ such that
\eqn{weakp}
$$
-\int_{\Omega_T} u \varphi_t \, dx \, dt + \int_{\Omega_T} \langle a(x,t,Du), D\varphi\rangle\, dx\, dt =  \int_{\Omega_T} \varphi\, d \mu
$$
holds whenever $\varphi \in C^\infty_0(\Omega_T)$ has compact support in $\Omega_T$. When referring to Cauchy-Dirichlet problems of the type \rif{parpro}, while the lateral boundary condition can be formulated by prescribing the inclusion $u \in L^1(-T, 0; W^{1,1}_0(\Omega))$, the initial boundary one $u(x,-T)\equiv 0$ is understood in the $L^1$-sense, that is prescribing that $$
    \lim_{h\downarrow 0} \frac{1}{h}\int_{-T}^{-T+h} \int_\Omega |u(x,t)|\, dx \, dt=0.
$$
In the parabolic setting a convenient \ap slicewise" reformulation of \rif{weakp} is given by mean of so-called Steklov averages; in fact, for $h >0$ and $t \in [-T, 0)$ we define
\eqn{steky}
$$
u_h(x,t):= \left\{
\begin{array}{ccc}
\displaystyle \frac{1}{h}\int_{t}^{t+h}  u(x,\tilde t)\, d\tilde t  & \mbox{if} & \ t+ h < 0 \\[3 pt]
0  & \mbox{if} & \ t+ h > 0\,.
\end{array}\right.
$$
With such a notation we have - see \cite[Chapter 2]{D} - that when $\mu \in L^1(\Omega_T)$ the slicewise equality
\begin{equation}\label{steky2}
    \int_{\Omega}\Big( \partial_tu_h\varphi +\langle [a(\cdot ,t,Du)]_h,D\varphi\rangle\Big)\, dx=\int_{B}\varphi\mu_h\, dx
   \end{equation}
holds whenever $\varphi \in C^\infty_0(\Omega)$ has compact support in $\Omega$, and for a.e. $t \in (-T, 0)$. As the reader will recognize we use both the notations $w_t$ and $\partial_t w$ for the time derivative of a function $w$.

\section{Elliptic estimates}\label{ellipticse}
This section contains the proof of estimate \rif{mainestx}, therefore we shall argue under the assumptions of Theorem \ref{mainx}, while $u  \in C^1(\Omega)$ denotes the solution identified in Theorem \ref{mainx}.
\subsection{Basic preliminaries.} Let us recall a few basic strict monotonicity properties of the vector field $z \mapsto a(\cdot, z)$ under the assumption \rif{asp}$_2$; for an absolute constant $c\equiv c(n,p,\ratio)\geq 1$ the inequality
With $s \geq
0$ being the number appearing in \rif{asp}, we define \eqn{mon1}
$$
c^{-1}(|z_1|^2+|z_2|^2 +s^2)^{\frac{p-2}{2}}|z_2-z_1|^2\leq
\langle a(x,z_2)-a(x,z_1), z_2-z_1\rangle
$$
holds whenever $z_1,z_2, \er^n$ and $x\in \Omega$. In particular since $p\geq 2$ the previous inequality implies
Indeed using \rif{asp}$_2$ is standard to see that the inequality
\eqn{mon3}
$$
c^{-1}|z_2-z_1|^p\leq
\langle a(x,z_2)-a(x,z_1), z_2-z_1\rangle \,.
$$
The next result on Reverse H\"older type inequalities will be useful in the sequel.
\begin{lemma}\label{revq}
Let $g : A \to \er^k$ be a integrable map such that
$$
\left(\mean{B_R} |g|^{\chi_0} \, dx \right)^{\frac{1}{\chi_0}}\leq c
\mean{B_{2R}} |g| \, dx
$$
holds whenever $B_{2R} \subseteq A$, where $A \subseteq
\er^n$ is an open subset, and $\chi_0>1$, $c \geq
0$.
Then, for every $t \in (0,1]$ and $\chi \in (0,\chi_0]$ there exists a constant $c_0 \equiv
c_0(n,c,t)$ such that, for every $B_{2R} \Subset A$ it holds that
\eqn{valida}
$$
\left(\mean{B_R} |g|^{\chi} \, dx \right)^{\frac{1}{\chi}}\leq
c_0 \left(\mean{B_{2R}}
|g|^t \, dx \right)^{\frac{1}{t}} \;.
$$
The assertion also holds when $A \subseteq
 \er^{n+1}$, considering backward parabolic cylinders of the type in \trif{caloricc}, instead of standard balls.
\end{lemma}

\subsection{A decay estimate below the natural growth exponent.}\label{fre}
Here we prove a decay estimate for solutions to certain homogeneous equations which differs from more standard ones in that the exponents involved are smaller than those typically used - the \ap natural growth exponents"; this estimate could be of its own interest. Let us consider an energy solution $v \in W^{1,p}(A)$ to the homogeneous equation
\eqn{hom}
$$
\divo\, a(x_0, Dv)=0\,,
$$
where $x_0 \in \Omega$ and $A\subseteq
 \Omega$ is a sub-domain of $\Omega$, in other words we consider the vector field from Theorem \ref{mainx} \ap frozen" at a point. 
\begin{lemma}\label{bbbb}Let $w \in W^{1,2}(A)$ be a solution of the linear equation of the type $$\divo (\tta(x)Dw)=0\,,$$ where the matrix $\tta(x)$ has measurable entries and satisfies the following elipticity and growth bounds, for every $\lambda \in \er^n$:
$$
c_*|\lambda|^2 \leq \langle \tta(x)\lambda, \lambda\rangle\,, \qquad |\tta(x)|\leq c_{**}\,, \qquad \mbox{for some}\ c_* \in (0,1)\ \mbox{and}\  c_{**}\geq 1\,.
$$
Then there exists constants $c \geq 1$ and $\beta_0 \in (0,1]$, both depending only on $n,c_{**}/c_*$, such that the following estimate holds whenever $B_\varrho\subset B_R \subseteq A$ are concentric balls:
\eqn{sottosti0}
$$
\mean{B_\varrho} |w-(w)_{B_{\varrho}}| \,dx\leq c \left(\frac{\varrho}{R}\right)^{\beta_0}\! \mean{B_R} |w-(w)_{B_{R}}|\, dx\,.
$$
\end{lemma}
\begin{proof} This result is a rather standard consequence of DeGiorgi's theory for linear elliptic equations. The statement follows directly from \cite[Corollary 1,5]{Lieb2} or by the proof of \cite[Theorem 7.7]{G}, see also \cite[(7.45)]{G}, taking into account the specific equations we are dealing with; in particular we may take $\chi=0$  in \cite{G}.
\end{proof}
\begin{theorem}\label{sotto} Let $v \in W^{1,p}(A)$ be a weak solution to \trif{hom} under the assumptions \trif{asp}. Then there exist constants $\beta \in (0,1]$ and $c\geq 1$, both depending only on $n,p,\ratio$,
such that the estimate
\eqn{sottosti}
$$
\mean{B_\varrho} |D v-(D v)_{B_{\varrho}}| \,dx\leq c \left(\frac{\varrho}{R}\right)^{\beta}\! \mean{B_R} |D v-(D v)_{B_{R}}|\, dx
$$
holds whenever $B_{\varrho}\subseteq B_{R}\subseteq A$ are concentric balls.
\end{theorem}
\begin{proof}
We shall use some iteration techniques \cite{DM, Lieb2, DM2}; in particular, we shall use the ideas in the important paper \cite{Lieb2}. The proof is divided in several steps. In the following $B(x_0,R)\equiv B_R$ is a fixed ball as considered in the statement of the Theorem. We recall that standard regularity gives that $Dv \in L^{\infty}_{\loc}(A)$; moreover without loss of generality we can assume that $B_R \Subset A$ so that in the following we have that $Dv \in L^{\infty}(B_R)$. From now on, for the rest of the proof, all the balls considered will have the same center $x_0$ and $r$ will denote a positive radius such that $r \leq R$. Accordingly, we set
$$
E(r):=\mean{B_r} |D v-(D v)_{B_{r}}| \,dx\,,\qquad \qquad M(r):=\max_{1\leq \xi\leq n} \sup_{B_r} |D_\xi v|\,.
$$
\indent {\em Step 1: Regularization.} In the following we shall consider the non-degenerate case $s>0$, but we shall prove estimates which do not depend on the parameter $s$. The case $s=0$ can be deduced by approximating original solutions by solutions to equations of similar type satisfying $s>0$, via a completely standard approximation procedure; see for instance \cite{df, Manth}.

{\em Step 2: The fundamental alternative.} This goes as follows - see \cite{df, Manth, Lieb2}: With $B_{2r} \subseteq B_{R}$ and $\mu_0 \in (0,1)$, clearly one of the following three alternatives must hold:
\eqn{dalt}
$$\left\{
\begin{array}{cc}
|\{D_\xi v < M(2r)/2\}\cap B_{2r}| \leq \mu_0 |B_{2r}| & \mbox{for some}\ \xi \\[5 pt]
|\{ D_\xi v > - M(2r)/2\}\cap B_{2r}| \leq \mu_0 |B_{2r}| & \mbox{for some}\ \xi \\ [5 pt]
\hspace{-1cm}|\{ D_\xi v < M(2r)/2\}\cap B_{2r}|, & \\ [5 pt]
\qquad \qquad |\{ D_\xi v > - M(2r)/2\}\cap B_{2r}| > \mu_0 |B_{2r}| & \mbox{for all}\ \xi     \,.
\end{array}
\right.
$$
The crucial and well-established regularity property of solutions - see for instance \cite{df, Lieb2, Manth} - is that there exist {\em universal numbers} $\mu_0, \eta \equiv \mu_0, \eta (n,p,\ratio) \in (0,1)$ - i.e. independent of the vector field considered $a(x_0,z)$ and of the solution $v$ -  such that if one of the first two alternatives in \rif{dalt} holds then
\eqn{uell}
$$
|Dv|\geq M(2r)/4 \qquad \qquad  \mbox{in}\ B_{r}
$$
otherwise, if the third possibility from \rif{dalt} occurs, we have that
\eqn{uell2}
$$
M(r)\leq \eta M(2r)\,.
$$
All in all we have that either \rif{uell} or \rif{uell2} happens. The proof now consists of showing that combining these two alternatives \rif{sottosti} follows. We recall that under the present assumptions we have that $v \in W^{2,2}_{\loc}(A)$ and each component $D_\xi v\equiv w$ satisfies the following differentiated equation:
\eqn{diffeqee}
$$
\divo\, (\tta(x)Dw)=0 \qquad \qquad  \qquad \tta_{i,j}(x) := (a_{i})_{z_j}(x_0,Dv(x))\,.
$$
Moreover, standard a priori estimates for $p$-Laplacean type equations as those considered in \rif{hom} provide a constant $c_{g}\equiv c_{g}(n, p, \ratio)$ such that
\eqn{ssuu}
$$
\sup_{B_r} |Dv| \leq c_{g} \mean{B_{2r}} (|Dv| + s) \, dx \qquad \mbox{holds for every}\, \, B_{2r}\subseteq B_{R}\,.
$$
\indent {\em Step 3: The non-degenerate case I.} In this case we assume that $s > 2\sqrt{n}M(r)$ for a certain radius $r\leq R$; then the matrix $\tta(x)$ in \rif{diffeqee} satisfies on $B_r$ the bounds
$$
\nu^{-1}  s^{p-2}|\lambda|^2 \leq \langle \tta(x)\lambda, \lambda\rangle\,, \qquad  |\tta(x)|\leq L2^{p-2}s^{p-2} \qquad \mbox{for every} \ \lambda \in \er^n\,.
$$
Therefore the matrix $\tta(x)$ defined in \rif{diffeqee} satisfies the assumptions of Lemma \ref{bbbb}, that  applied to every component $D_\xi v\equiv w$ of the gradient gives
\eqn{simpledecay}
$$
\mean{B_\varrho} |D v-(D v)_{B_{\varrho}}| \,dx\leq c_{d} \left(\frac{\varrho}{r}\right)^{\beta_0}\! \mean{B_r} |D v-(D v)_{B_{r}}|\, dx
$$
where $c_{d}\geq 1, \beta_0\in (0,1)$ depend only $n,p,\ratio$.

{\em Step 4: Choice of the constants.}
We first take a positive $H_1 \in \en$ such that
\eqn{ch0}
$$8\sqrt{n}c_{g}\eta^{H_1-1} \leq 1\,.$$ This fixes $H_1$ as a quantity depending only on $c_{g}, \eta$ and therefore ultimately on $n,p,\ratio$. In turn we determine another integer $K_1$ such that
\eqn{ch}
$$
 2^{nH_1+2} \eta^{K_1} \leq 1\,,\qquad \qquad c_{d}2^{-K_1\beta_0+nH_1+2}\leq 1
$$
hold. This fixes $K_1$ as a quantity depending on $\eta, H_1, c_{d}$, and therefore ultimately on $n,p,\ratio$.

{\em Step 5: The degenerate case I.} We consider the following situation: there exists a radius $t \leq R$ such that \rif{uell2} happens to hold whenever $r = t/2^{i}$ and $1\leq i \leq H_1 \in \en$, and finally we assume that
both
\eqn{bothh}
$$|(Dv)_{B_t}|\leq 2\sqrt{n}M(2^{-H_1}t)\qquad \mbox{and}\qquad s \leq 2\sqrt{n}M(2^{-H_1}t) $$ hold. We first notice that $E(2^{-H_1}t)\leq 2\sqrt{n}M(2^{-H_1}t)$, while
iterating \rif{uell2} yields $$M(2^{-H_1}t)\leq \eta^{H_1-1}M(t/2)\,.$$ On the other hand we have that
\begin{eqnarray*}
\nonumber M(2^{-H_1}t)& \leq & \eta^{H_1-1}M(t/2) \\ \nonumber & \leq &  c_{g}\eta^{H_1-1}\mean{B_{t}} (|Dv| + s) \, dx \\
\nonumber &\leq & c_{g}\eta^{H_1-1}\mean{B_{t}} |Dv- (Dv)_{B_t}|  \, dx+  c_{g}\eta^{H_1-1}[|(Dv)_{B_t}| + s]\\
&\leq & c_{g}\eta^{H_1-1}E(t)+  4\sqrt{n}c_{g}\eta^{H_1-1}M(2^{-H_1}t)\,,
\end{eqnarray*}
where we used \rif{ssuu}. Therefore we have
$
M(2^{-H_1}t)\leq 2\eta^{H_1-1}c_{g} E(t),
$
so that, again by \rif{ch0}, the final outcome is
\eqn{out1}
$$
E(2^{-H_1}t)\leq {\textstyle \frac12} E(t)\,.
$$
\indent {\em Step 6: The degenerate case II.} Continuing the reasoning of the previous step, we assume that there exists a radius $t \leq R$ such that \rif{uell2} happens to hold whenever $r = t/2^{i}$ and $1\leq i \leq H_1+K_1 \in \en$, assuming also, in alternative to \rif{bothh},
that {\em at least} one of the following inequalities holds: \eqn{nbothh}
$$|(Dv)_{B_t}|> 2\sqrt{n}M(2^{-H_1}t)
\qquad \mbox{and}\qquad s> 2\sqrt{n}M(2^{-H_1}t)
\,.$$ In the case the first inequality in \rif{nbothh} holds we have that also $$|Dv-(Dv)_{B_{t}}|>\sqrt{n}M(2^{-H_1}t)\qquad \mbox{in} \  B_{2^{-H_1}t}\,;$$ therefore, using also \rif{ch}, we conclude as in \rif{out1}:
\begin{eqnarray}
\nonumber E(2^{-(H_1+K_1)}t) & \leq & 2\sqrt{n}M(2^{-(H_1+K_1)}t)\\ \nonumber &  \leq & 2\sqrt{n}\eta^{K_1}M(2^{-H_1}t)\\
\nonumber  &\leq & 2 \eta^{K_1}\mean{B_{2^{-H_1}t}}|Dv-(Dv)_{B_{t}}|\, dx\\ & \leq & 2^{nH_1+1} \eta^{K_1}\mean{B_{t}}|Dv-(Dv)_{B_{t}}|\, dx\nonumber \\ & \leq & {\textstyle \frac12} E(t)\,.\label{concdecay}
\end{eqnarray}
On the other hand, if $s> 2\sqrt{n}M(2^{-H_1}t)
$, we are in the situation of Step 3 with the choice $r\equiv  2^{-H_1}t$ and $\varrho\equiv 2^{-(H_1+K_1)}t$; therefore applying \rif{simpledecay} yields
$$
E(2^{-(H_1+K_1)}t) \leq c_{d}2^{-K_1\beta_0}E(2^{-H_1}t)
$$
while it also holds
$$
E(2^{-H_1}t)\leq 2\mean{B_{2^{-H_1}t}}|Dv-(Dv)_{B_{t}}|\, dx \leq 2^{nH_1+1}E(t)\,,
$$
so that, merging the last two estimates and using \rif{ch} we again conclude as in \rif{concdecay} with
$$
E(2^{-(H_1+K_1)}t) \leq c_{d}2^{-K_1\beta_0+nH_1+1}E(t)\leq {\textstyle \frac12} E(t)\,.
$$
\indent {\em Step 7: Summarizing Step 5 and Step 6.} Here we first note that the result of Step 3 holds if we replace $H_1$ by any larger integer since if \rif{ch0} holds for $H_1$, then so it does for any larger number. Therefore, looking at Step 6; if \rif{uell2} holds for $r = t/2^{i}$ and $1\leq i \leq H_1+K_1 \in \en$, and assuming also
both $$|(Dv)_{B_t}|\leq 2\sqrt{n}M(2^{-H_1-K_1}t)\qquad \mbox{and}\qquad s <2\sqrt{n}M(2^{-H_1-K_1}t)\,,$$ we can apply Step 5 with $H_1$ replaced by $H_1+K_1$. Therefore, summarizing the results of Steps 5 and 6 we have that: If \rif{uell2} holds for $r = t/2^{i}$ and $1\leq i \leq H:=H_1+K_1$ then
\eqn{basic00}
$$
E(\tau t)\leq {\textstyle \frac12} E(t)\,,\qquad \qquad \qquad \tau:=1/2^{H}\,.
$$
The crucial fact is that $H$, and therefore also $\tau$, only depends on $n,p,\ratio$.

{\em Step 8: Conclusion.} Let use define $\beta_1 := 1/H$ so that $\beta_1\equiv \beta_1(n,p,\ratio)\in (0,1]$ and $\tau^{\beta_1} = 1/2$, where $\tau \in (0,1/2]$ has been introduced in \rif{basic00}; we finally determine the exponent $\beta$ appearing in \rif{sottosti} by letting $\beta:=\min \{\beta_0,\beta_1\}$ where $\beta_0$ is the exponent from Lemma \ref{bbbb}. Moreover, with $\varrho \leq R$ as in the statement of the Theorem, we fix $k \in \en$ such that $\tau^{k+1} R < \varrho \leq \tau^{k} R$. Let us now define
$$
\mathbb S:= \{i \in \en \, : \, \trif{uell2}\  \mbox{holds for} \ r=R/2^i\,, i\geq 1\}\,.
$$
We argue on an alternative; the first case is when $\mathbb S = \en\setminus \{0\}$. Then either Step 5 or Step 6, and so Step 7, apply for every choice of the radius $t\equiv \tau^i R$ and $i \geq 0$, thus obtaining that
\eqn{decadi}
$$
 E(\tau^{i}R)\leq (1/2)^i E( R)= \tau^{i\beta_1} E( R)\leq \tau^{i\beta} E( R)
$$
holds for every $ i \in \en$. Taking into account the definition of $\beta_1$ we have
$$
E(\varrho)\leq 2\tau^{-n} E(\tau^kR)\leq c\tau^{k\beta}E(R)\leq c \left(\frac{\varrho}{R}\right)^\beta E(R)\,,
$$
so that \rif{sottosti} follows; note here that $c \equiv c(\tau)\equiv c(n,p,\ratio)$ by \rif{basic00}. In the other case we have that $\mathbb S \ne (\en\setminus \{0\}) $, and therefore if we set $m:= \min\,\big( (\en\setminus \{0\}) \setminus \mathbb S\big)$ then \rif{uell} implies that $|Dv| \geq M(R/2^{m-1})$ in the ball $B_{R/2^m}$. At this point we observe that in the ball $B_{R/2^m}$ the matrix $\tta(x)$ defined in \rif{diffeqee} satisfies the following uniform ellipticity bounds:
$$
c^{-1} [M(R/2^{m-1})+s]^{p-2}|\lambda|^2 \leq \langle \tta(x)\lambda, \lambda\rangle\,, \qquad  |\tta(x)|\leq c[M(R/2^{m-1})+s]^{p-2}
$$
for $c \equiv c(n,p,\ratio)>1$ and every choice of $\lambda \in \er^n$, and therefore we may apply Lemma \ref{bbbb} as already done in Step 3, thereby getting
\eqn{user00}
$$
E(r) \leq c\left(\frac{r}{R/2^m}\right)^{\beta_0} E(R/2^m) \qquad \qquad \mbox{for every} \ r \leq R/2^m\,.$$
Let now $\gamma \in \en$ be the unique non-negative integer such that $\gamma H <  m \leq  (\gamma+1)H$. We apply Step 7 exactly $\gamma$ times thereby getting - also when $\gamma=0$ - that
\eqn{iprimi}
$$
E(\tau^i R)\leq (1/2)^{i} E(R)\qquad \qquad \mbox{for every} \ i \leq \gamma\,.$$
Moreover, using the definition of $E(\cdot)$ and $\tau$ it immediately follows that
$$
E(R/2^m) \leq 2\mean{B_{R/2^m}}|Dv-(Dv)_{B_{\tau^\gamma R}}|\, dx\leq 2^{nH+1}E(\tau^\gamma R)\,,
$$
so that combining the last two inequalities gives
$$
E(R/2^m) \leq 2^{nH+1}2^{-\gamma} E(R)=c \, \tau^{\gamma \beta_1} E(R)\le c \, \tau^{\gamma \beta} E(R)
$$
with $c \equiv c(n,p,\ratio)$. On the other hand using \rif{user00} with $r = \tau^{\gamma+l}R\equiv 2^{-(\gamma+l)H}R$ and $l \in (\en\setminus \{0\})$, we gain
\begin{eqnarray*}
    E(\tau^{\gamma+\ell}R) &\leq & c \Big(\frac{\tau^{\gamma +\ell}R}{R/2^m}\Big)^{\beta_0} E(R/2^m)
    \\ & \le &  c\, 2^{H\beta_0} \tau^{\ell\beta_0} E(R/2^m)\\ &\leq & c\, \tau^{(\gamma+\ell) \beta} E(R)\, .
\end{eqnarray*}The last estimate and \rif{iprimi} in turn imply, for a suitable $c\equiv c (n,p,\ratio)$, that the inequality
$
 E(\tau^{i}R)\leq  c\tau^{i\beta} E( R)
$
holds for every $ i \in \en$. At this point we again conclude as after \rif{decadi}.
\end{proof}
\subsection{Comparison estimates.}\label{compsec} This section is devoted to the proof of a few suitable comparison estimates in which the density of the non-linear Wolff potential $\npma$ explicitly comes into the play. We now fix, for the rest of Section \ref{compsec}, a ball $B_{2R}\equiv B(x_0,2R)\subseteq
 \Omega$ with radius $2R$; we start defining $w \in u +
W^{1,p}_0(B_{2R})$ as the unique solution to the homogeneous Dirichlet problem
\eqn{Dirc1}
$$
\left\{
    \begin{array}{cc}
    \divo \ a(x,Dw)=0 & \qquad \mbox{in} \ B_{2R}\\
        w= u&\qquad \mbox{on $\partial B_{2R}$\,.}
\end{array}\right.
$$
\begin{lemma} \label{coco1} Under the assumption \trif{asp}$_2$, let $u \in W^{1,p}(\Omega)$ be as in Theorem \ref{mainx},
and $ w \in u +W^{1,p}_0(B_{2R})$
 as in \trif{Dirc1}. Then
 the following inequalities hold for a constant $c \equiv c(n,p,\nu)$:
\eqn{comp1}
$$
\mean{B_{2R}} |Du-Dw|\, dx \leq c\left[\frac{|\mu|(B_{2R})}{R^{n-1}}\right]^{\frac{1}{p-1}}
$$
\eqn{comp111}
$$
\mean{B_{2R}} |u-w|\, dx \leq c\left[\frac{|\mu|(B_{2R})}{R^{n-p}}\right]^{\frac{1}{p-1}}\,.
$$
\end{lemma}
\begin{proof} The proof revisit and modifies various comparison and truncation methods for measure data problems scattered in the literature; a chief reference here is the work of Boccardo \& Gall\"ouet \cite{BG1}. The proof is divided in three steps. We may without loss of generality assume that $2R=1$, that the ball in question is centered at the origin $B_{2R}\equiv B_1$, and finally that $|\mu|(B_1)= 1$; the scaling technique necessary to reduce to such a situation is reported in Step 3 below. 

{\em Step 1: The case $p>n$}. We test the weak
formulations of \rif{baseq} and \rif{Dirc1}$_1$ \eqn{weak2}
$$
\int_{B_1} \langle a(x,Du) -a(x,Dw),D \varphi \rangle\, dx =
\int_{B_1} \varphi \, d\mu \;,
$$
with $\varphi\equiv u-w$. Using
Morrey-Sobolev's embedding theorem we estimate the resulting right hand side as follows:
\begin{eqnarray*}\left| \int_{B_1}
(u-w) \, d\mu \right| & \leq & \sup_{B_1}\, |u-w|[|\mu|(B_1)] \\ & \leq & \sup_{B_1}\, |u-w| \leq  \frac{cp}{p-n}\left(\int_{B_1}|Du-Dw|^p\, dx \right)^{\frac{1}{p}}\,.
\end{eqnarray*}
The last inequality used together with \rif{weak2} with $\varphi = u-w$ and
\rif{mon3} yields
$$  \int_{B_1}| Du-Dw|^p\, dx
\leq
c\left(\int_{B_1}|Du-Dw|^p\, dx \right)^{\frac{1}{p}} \,,
$$
so that we first get
$ \| Du-Dw\|_{L^p(B_1)}\leq c$
and eventually
$ \| Du-Dw\|_{L^1(B_1)}\leq c$
that is \rif{comp1} - when $B_{2R}\equiv B_1$ and $|\mu|(B_1)= 1$.

{\em Step 2: The case $2\leq  p \leq n$}. For $k \geq  0$ denoting an integer, we define the following truncation operators
\eqn{troncamenti}
$$
    T_k(s):=\max\{ -k, \min\{s,k\}\}\quad\mbox{and}\quad
    \Phi_k(s):= T_1(s-T_k(s))\, ,
$$
 defined for $s \in \er$.
We test the weak formulation
\rif{weak2} by $\varphi\equiv T_k(u-w)$.
By \rif{mon3}, setting $D_k:=\{x \in B_1  :  |u(x)-w(x)|\leq k\}$, we obtain, with $c\equiv c(n,p,\nu)$,
\eqn{boc01}
$$
\int_{D_k} | Du-Dw|^p\, dx
\leq ck\,.
$$ Similarly, testing \rif{weak2} with $\varphi \equiv
\Phi_k(u-w)$ yields \eqn{boc1}
$$
\int_{C_k}| Du-Dw|^p \, dx \leq c\,,
$$
where this time
$C_k:=\{x \in B_1  : k  < |u(x)-w(x)|
\leq k+1\}$ and $c\equiv c(n,p,\nu)$. By H\"older's inequality, \rif{boc1} and
the very definition of $C_k$, we find
\begin{eqnarray}
\nonumber \int_{C_k} | Du-Dw|\, dx
&\leq&
|C_k|^{\frac{p-1}{p}} \left(\int_{C_k} | Du-Dw|^p\,
dx\right)^{\frac{1}{p}} \notag  \\
& \leq&c
|C_k|^{\frac{p-1}{p}}\leq c
k^{-\frac{n(p-1)}{p(n-1)}}\left(\int_{C_k}
|u-w|^{\frac{n}{n-1}}\, dx\right)^{\frac{p-1}{p}}\;.\label{gal}
\end{eqnarray}
We have of course used the elementary estimate
$$
|C_k|\leq \frac{1}{k^{\frac{n}{n-1}}}\int_{C_k} |u-w|^{\frac{n}{n-1}}\, dx\;.
$$
Using \rif{gal} with \rif{boc01}, H\"older's inequality for sequences, and finally
Sobolev's embedding theorem, with $k_0$ being a fixed positive integer we have
\begin{eqnarray}
\nonumber  \int_{B_1} | Du-Dw| \,dx &=&  \int_{D_{k_0}} (\cdot\cdot \cdot)\, dx + \sum_{k=k_0}^\infty\int_{C_{k}} (\cdot\cdot \cdot)\, dx
\\ \nonumber
   & \leq & ck_0^{\frac{1}{p}} + c\sum
_{k=k_0}^{\infty}
k^{-\frac{n(p-1)}{p(n-1)}}\left(\int_{C_k}
|u-w|^{\frac{n}{n-1}}\, dx\right)^{\frac{p-1}{p}}\\ \nonumber
   & \leq & ck_0^{\frac{1}{p}} + c\left[\sum
_{k=k_0}^{\infty}\frac{1}{k^{\frac{n(p-1)}{n-1}}}\right]^{\frac{1}{p}}\left(\int_{B_1}
|u-w|^{\frac{n}{n-1}}\, dx\right)^{\frac{p-1}{p}}\nonumber \\  &\leq  &
ck_0 + cH(k_0)\left(\int_{B_1} |Du-Dw|\,
dx\right)^{\frac{(p-1)n}{(n-1)p}}\,.\label{need}
\end{eqnarray}
In the last lines we have obviously set $$H(k_0):=\left[\sum
_{k=k_0}^{\infty}\frac{1}{k^{\frac{n(p-1)}{n-1}}}\right]^{\frac{1}{p}}\,,$$ while $c$ depends on $n,p,\nu$; note that $$p\geq 2\Longrightarrow \frac{n(p-1)}{n-1}>1$$ so that $H(k_0)$ is always finite and satisfies $H(k_0)\to 0$ when $k_0 \to \infty$.
Now, if $p<n$ then $$\frac{(p-1)n}{(n-1)p} < 1\,;$$ therefore we take $k_0=1$ in \rif{need}
and applying
Young's inequality in \rif{need} we find
\eqn{ineVV}
$$\|Du-Dw\|_{L^1(B_1)}\leq c\,,$$ that is \rif{comp1} when $2R=1$ and $|\mu|(B_1)= 1$. When $p=n$ we
choose $k_0 \equiv k_0 (n,p,\nu)$ large enough in order to have $cH(k_0)=1/2$ in
\rif{need} and reabsorb the last integral on the right hand side, so that \rif{ineVV} follows again.

{\em Step 3: Scaling procedures}. We first reduce to the case $B_{2R}\equiv B_1$ by a standard scaling argument, i.e. letting
$$\tilde{u}(y) := \frac{u(x_0+2Ry)}{2R}\,, \qquad \tilde{w}(y) := \frac{w(x_0+2Ry)}{2R}\,,$$ and
$$\tilde{a}(y,z) := a(x_0+2Ry,z)\,, \qquad \tilde{\mu}(y) := 2R\mu(x_0+2Ry)$$ for $y \in B_1$ so that $-\divo\ \tilde{a}(y,D\tilde{u}) =
\tilde{\mu}$ and $\divo\ \tilde{a}(y,D\tilde{w}) =0$ hold. At this point one writes estimate \rif{comp1} for $\tilde{u}$ and $\tilde{w}$ and then scales back.  To reduce to the case $|\mu|(B_1)=1$ we adopt another scaling: this time we define $A :=
[|\mu|(B_1)]^{1/(p-1)}$, and we may assume $A>0$ otherwise $u\equiv w$ and \rif{comp1} follows trivially by the strict monotonicity of the operator $a(\cdot)$. We define the new solutions
$\bar{u}:=A^{-1}u$, $\bar{w}:=A^{-1}w,$ the new datum
$\bar{\mu}:=A^{1-p}\mu$, and the new vector field
$\bar{a}(x,z):=A^{1-p}a(x,Az).$ Therefore we have $\divo \
\bar{a}(x,D\bar{u})=\bar{\mu}$, $\divo\
\bar{a}(x,D\bar{w})=0$, in the weak sense. We make sure that we
can apply the result in Step 2. Trivially
$|\bar{\mu}|(B_1)=1$ and moreover it is easy to see that the
vector field $\bar{a}(x,z)$ satisfies \rif{asp} with $s$ replaced
by $s/A\geq 0$. Therefore \rif{comp1} holds in the
form
$$
\int_{B_1}|D\bar{u}-D\bar{w}|\, dx\leq c_2$$ with $c_2 \equiv
c_2(n,p,\nu)$. At this point we find back \rif{comp1} just using the definitions of $\bar u, \bar w$. As for \rif{comp111}, this is a consequence of \rif{comp1} via the use of Sobolev inequality:
$$ \mean{B_{2R}} |u-w|\, dx \leq cR\mean{B_{2R}} |Du-Dw|\, dx\leq    c\left[\frac{|\mu|(B_{2R})}{R^{n-p}}\right]^{\frac{1}{p-1}}\,.
$$The proof is now complete\end{proof}
\begin{remark} In the proof of the gradient estimate \rif{mainestx} we shall only need \rif{comp1}, while \rif{comp111} will be needed only later, in the different context of Remark \ref{alternative} below.
\end{remark}
After introducing $w$ in \rif{Dirc1} we similarly define $v \in w +
W^{1,p}_0(B_R)$, on the concentric smaller ball $B_R\equiv B(x_0,R)$, as the unique solution to the homogeneous Dirichlet problem with frozen coefficients
\eqn{Dirc12}
$$
\left\{
    \begin{array}{cc}
    \divo \ a(x_0,Dv)=0 & \qquad \mbox{in} \ B_R\\
        v=w  &\qquad \mbox{on $\partial B_R$\,.}
\end{array}\right.
$$
Next, we state
a standard comparison result; we report its proof for completeness.
\begin{lemma} \label{coco6} Under the assumptions of Theorem \ref{mainx},
with $w$ as in \trif{Dirc1} and $v$ as in \trif{Dirc12}, there exists a constant
$c\equiv c(n,p,\ratio)$ such that the following inequality holds:
 \eqn{baba3}
$$
\mean{B_{R}} |Dw-Dv|^{p} \, dx\leq c
[L_1\omega(R)]^{2}\mean{B_{R}}(|Dw|+s)^p\, dx \;.
$$
\end{lemma}
\begin{proof} We note that by the growth condition on $a$ in \rif{asp}$_1$ and by \rif{mon3}, we also have the following coercivity condition:
\eqn{veramon}
$$
c^{-1}(|z|^2+s^2)^{\frac{p-2}{2}}|z|^2-cs^p \leq \langle
a(x,z),z\rangle, \qquad  c \equiv c(n,p,\ratio)\geq 1\,,
$$
which easily follows by Young's inequality. Using \rif{veramon}, that is testing \rif{Dirc12}$_1$ with $w-v$, and then
 \rif{asp}$_2$ and Young's inequality, gives
the energy bound
\eqn{coecoe}
$$
\int_{B_R} |Dv|^p\, dx \leq
c\int_{B_R} (|Dw|+s)^p\, dx,
$$
which holds for a constant $c$ depending on $n,p,\ratio$. In turn, using \rif{mon1}, the fact that both $v$ and $w$ are solutions, \rif{asp}$_3$ and again Young's inequality, we have
\begin{align*}
\int_{B_{R}}
(s^2&+|Dv|^2+|Dw|^2)^{\frac{p-2}{2}}|Dw-Dv|^2\,
dx\\
&\leq c\int_{B_{R}} \langle a(x_0,Dw)-a(x_0,Dv),
Dw-Dv\rangle \, dx\\& =c\int_{B_{R}} \langle
a(x_0,Dw)-a(x,Dw), Dw-Dv\rangle \, dx\\&
\leq cL_1\omega(R)
\int_{B_{R}}(|Dv|^2+|Dw|^2+s^2)^\frac{p-1}{2}|Dw-Dv|\,
dx \\ &\leq
\frac{1}{2}\int_{B_{R}}
(|Dv|^2+|Dw|^2+s^2)^\frac{p-2}{2}|Dw-Dv|^2\, dx\\
&\qquad\qquad\qquad
+ c[L_1\omega(R)]^2\int_{B_{R}}
(|Dv|+|Dw|+s)^{p}\, dx\;.
\end{align*}
Therefore we gain
$$
\int_{B_{\bar{R}}} (s^2+|Dv|^2+|Dw|^2)^{\frac{p-2}{2}}|Dw-Dv|^2 \, dx \leq c[L_1\omega(R)]^2\int_{B_{R}}(|Dv|+|Dw|+s)^{p}\, dx,
$$
and \rif{baba3} follows by last inequality recalling that $p\geq 2$.
\end{proof}
\begin{lemma}\label{compx} Let $u$ be as in Theorem \ref{mainx}, and let $w \in u +  W^{1,p}_0(B_{2R})$ and $v\in w +  W^{1,p}_0(B_R)$ be defined in \trif{Dirc1} and \trif{Dirc12}, respectively. Then for a constant $c$ depending only on $n,p,\ratio$, it holds that
\begin{eqnarray}
\mean{B_{R}} |Du-Dv|\, dx\nonumber &\leq & c\left\{1+[L_1\omega(R)]^{\frac{2}{p}}\right\}\left[\frac{|\mu|(B_{2R})}{R^{n-1}}\right]^{\frac{1}{p-1}}\\ && \qquad + c[L_1\omega(R)]^{\frac{2}{p}}\mean{B_{2R}}(|Du|+s)\, dx\,.\label{comp12}
\end{eqnarray}
\end{lemma}
\begin{proof}
We start proving the following inequality:
\eqn{inter1}
$$\mean{B_{R}} |Dw-Dv|\, dx  \leq  c [L_1\omega(R)]^{\frac{2}{p}}\mean{B_{2R}} (|Dw|+ s)\, dx\,,
$$
with $c \equiv c(n,p,\ratio)$. Keeping \rif{veramon} in mind we may apply Gehring's lemma in the version presented in \cite[Chapter 6]{G}, finding there exists a constant $\chi_0\equiv \chi_0(n,p,\ratio)>1$ such that the reverse H\"older type inequality
$$
\left(\mean{B_{\varrho/2}} (|Dw|+s)^{\chi_0p}\, dx \right)^\frac{1}{\chi_0}\leq c \mean{B_{\varrho}} (|Dw|+s)^{p}\, dx
$$
holds whenever $B_\varrho \subseteq B_{{2R}}$, for a constant $c$ depending only on $n,p,\ratio$. In turn, applying Lemma \ref{revq} with $g\equiv (|Dw|+s)^p$, leads to establish that also
\eqn{rev}
$$
\left(\mean{B_{R}} (|Dw|+s)^{p}\, dx \right)^\frac{1}{p}\leq c \mean{B_{2R}} (|Dw|+s)\, dx
$$
holds. Using now \rif{baba3} and previous inequality we estimate as follows:
\begin{eqnarray*} &&\mean{B_{R}} |Dw-Dv|\, dx  \leq  c \left(\mean{B_{R}}  |Dw-Dv|^{p}\, dx\right)^{\frac{1}{p}}\\
&&\quad   \leq  c[L_1\omega(R)]^{\frac{2}{p}} \left( \mean{B_{R}}(|Dw|+s)^p\, dx\right)^{\frac{1}{p}}\leq  c[L_1\omega(R)]^{\frac{2}{p}} \mean{B_{2R}} (|Dw|+s)\, dx
\end{eqnarray*}
and therefore \rif{inter1} follows. Using \rif{inter1} and \rif{comp1} now yields
\eqn{comp123}
$$ \mean{B_{R}} |Du-Dv|\, dx  \leq  c\left[\frac{|\mu|(B_{2R})}{R^{n-1}}\right]^{\frac{1}{p-1}}+ c[L_1\omega(R)]^{\frac{2}{p}}\mean{B_{2R}}(|Dw|+s)\, dx\,.
$$
In order to estimate the last integral in \rif{comp123} we simply use \rif{comp1}
as follows:
\begin{eqnarray*}
\mean{B_{2R}}|Dw|\, dx& \leq &  \mean{B_{2R}}|Du|\, dx +  \mean{B_{2R}}|Du-Dw|\, dx\\
 & \leq  &  \mean{B_{2R}}|Du|\, dx + c(n,p,\nu) \left[\frac{|\mu|(B_{2R})}{R^{n-1}}\right]^{\frac{1}{p-1}}\,.
\end{eqnarray*}
Using the last inequality in combination with \rif{comp123} we conclude with \rif{comp12}.
\end{proof}
\subsection{Proof of the main estimates}\label{ellell}
We start with a technical lemma.
\begin{lemma}\label{compxxx} Let $u$ be as in Theorem \ref{mainx}, then there exists constants $\beta \in (0,1]$ and $c,c_1\geq 1$, all depending only on $n,p,\ratio$, and a positive radius $R_1 \equiv R_1 (L_1,\omega(\cdot))$, such that the following estimate holds whenever $B_{\varrho} \subseteq B_R \subseteq B_{2R} \subseteq
 \Omega$ are concentric balls with $R\leq R_1$:
\begin{eqnarray}
\nonumber && \mean{B_\varrho} |D u-(D u)_{B_{\varrho}}|\, dx \leq  c_1 \left(\frac{\varrho}{R}\right)^{\beta} \mean{B_{2R}} |D u-(D u)_{B_{2R}}|\, dx \\ &&\qquad + c\left(\frac{R}{\varrho}\right)^{n}\left[\frac{|\mu|(B_{2R})}{R^{n-1}}\right]^{\frac{1}{p-1}}+ c\left(\frac{R}{\varrho}\right)^{n}[L_1\omega(R)]^{\frac{2}{p}}\mean{B_{2R}}(|Du|+s)\, dx\,.\label{CComp}
\end{eqnarray}
Moreover, in the case the vector field $a(\cdot)$ is independent of the variable $x$, the previous inequality holds without any restriction on $R$.
\end{lemma}
\begin{proof} Starting by $B_{2R}$ we define the comparison functions $v$ and $w$ as in \rif{Dirc12} and \rif{Dirc1}, respectively, and then we compare $Du$ and $Dv$ by mean of \rif{comp12}, using \rif{sottosti} as basic reference estimate for $v$, eventually transferred to $u$:
\begin{eqnarray}
&& \nonumber \mean{B_\varrho} |D u-(D u)_{B_{\varrho}}|\, dx \leq 2
 \mean{B_\varrho} |D u- (D v)_{B_{\varrho}}|\, dx\\
&& \nonumber \leq 2\mean{B_\varrho} |D v-(D v)_{B_{\varrho}}|\, dx + 2\mean{B_\varrho} |D u- D v|\, dx
\\
&& \nonumber \leq  c\left(\frac{\varrho}{R}\right)^\beta\mean{B_R} |D v-(D v)_{B_{R}}|\, dx + c\left(\frac{R}{\varrho}\right)^n
\mean{B_R} |Du-Dv|\, dx\\
&& \nonumber \leq  c\left(\frac{\varrho}{R}\right)^\beta\mean{B_R} |D u-(D u)_{B_{R}}|\, dx + c\left(\frac{R}{\varrho}\right)^n
\mean{B_R} |Du-Dv|\, dx\\
&& \leq  2^{n+1}c\left(\frac{\varrho}{R}\right)^\beta\mean{B_{2R}}
|D u-(D u)_{B_{2R}}|\, dx + c\left(\frac{R}{\varrho}\right)^n
\mean{B_R} |Du-Dv|\, dx\,.\label{bchain}
\end{eqnarray}
Notice that from the second-last to the last line we estimated as follows:
\begin{eqnarray}
 &&\nonumber  \mean{B_{R}} |D u-(Du)_{B_{R}}|\, dx \leq \mean{B_{R}} |D u-(D u)_{B_{2R}}|\, dx \\
&&  \qquad  + |(D u)_{B_{R}}- (D u)_{B_{2R}}|\leq  2^{n+1}\mean{B_{2R}} |D u-(D u)_{B_{2R}}|\, dx\,.\label{stimalike}
\end{eqnarray}
In order to get \rif{CComp} it is now sufficient to estimate the last integral in \rif{bchain} by mean of \rif{comp12}
and then to take the radius $R_1$ such that $L_1\omega(R_1)\leq 1$. Needless to say, in the case $a(\cdot)$ does not depend on the variable $x$ we can take $\omega(\cdot)\equiv 0$ and therefore no restriction on $R$ is needed.
\end{proof}
\begin{proof}[Proof of Theorem \ref{mainx}] The proof is divided in three steps.
In what follows all the radii considered will be smaller than a certain radius $\tilde{R}$:
\eqn{raggio}
 $$
 R \leq \tilde{R}
 $$
whose final values will be determined towards the end of the proof, i.e. we shall decrease the values of $\tilde{R}$ several times according to our needs, but always in such a way that the resulting determination of $\tilde{R}$ will be still depending only on $ n, p,\ratio, L_1,\omega(\cdot)$. This will finally give the radius $\tilde{R}$ in the statement of Theorem \ref{mainx}. From the proof it will be clear that such a restriction is necessary only in the case the vector field $a(\cdot)$ depends on the variable $x$, while all the inequalities holds with no restriction on $R$ when $a\equiv a(z)$; the final outcome is that estimate \rif{mainestx} holds for every ball $B(x_0,2R)\subseteq \Omega$. This observation also clarifies the last assertion in the statement of Theorem \ref{mainx}. Initially, we shall take $\tilde{R} \leq R_1$ where the radius $R_1\equiv R_1(L_1, \omega(\cdot))$ has been determined in Lemma \ref{compxxx} above. In the following we shall consider several constants, in general depending {\em at least} on the parameters $n,p,\ratio$; relevant dependence on additional parameters will be emphasized.

{\em Step 1: Basic dyadic sequence.} Referring to estimate \rif{CComp}, we select an integer $H \equiv H(n,p,\ratio)\geq 1$ such that
\eqn{sceltasmall}
$$
c_1\left(\frac{1}{H}\right)^{\beta} \leq \frac{1}{4}\,.
$$
We notice that the dependence of $H$ upon $n,p,\ratio$ comes from the similar dependence of $\beta$ and $c_1$ presented in Lemma \ref{compxxx}.
Applying \rif{CComp} on arbitrary balls $B_\varrho \equiv B_{R/2H} \subseteq B_{R/2} \subset
 B_R$ and using the fact that $\omega(\cdot)$ is non-decreasing we gain
\begin{eqnarray}
&&\nonumber \mean{B_{R/2H}} |D u- (D u)_{B_{R/2H}}|\, dx \leq  \frac{1}{4} \mean{B_{R}} |D u- (D u)_{B_{R}}|\, dx\\
&& \qquad \qquad \qquad  + c\left[\frac{|\mu|(B_{R})}{R^{n-1}}\right]^{\frac{1}{p-1}}+ c_2[L_1\omega(R)]^{\frac{2}{p}}\mean{B_{R}}(|D u|+s)\, dx\,,\label{CCompdopo0}
\end{eqnarray}
where $c,c_2$ depends only on $n,p,\ratio, H$ and therefore ultimately on $n,p,\ratio$.
By the elementary estimation
$$
\mean{B_{R}}(|D u|+s)\, dx\leq \mean{B_{R}}|D u-(D u)_{B_{R}}|\, dx + |(D u)_{B_{R}}|+s
$$
estimate \rif{CCompdopo0} turns to
\begin{eqnarray}
\nonumber&& \mean{B_{R/2H}} |D u- (D u)_{B_{R/2H}}|\, dx   \leq  \left(\frac{1}{4} +
c_2[L_1\omega(R)]^{\frac{2}{p}}\right)\mean{B_{R}} |Du- (D u)_{B_{R}}|\, dx\\
&&\hspace{4cm}  + c\left[\frac{|\mu|(B_{R})}{R^{n-1}}\right]^{\frac{1}{p-1}}+ c[L_1\omega(R)]^{\frac{2}{p}}\left(|(Du)_{B_R}|+s\right).\label{CCompdopo1}
\end{eqnarray}
We start reducing the value of $\tilde{R}$: we take $\tilde{R}$, depending only on $n, p,\ratio, L_1$ and $\omega(\cdot)$, small enough in order to get
\eqn{sceltasmall2}
$$
c_2[L_1\omega(\tilde{R})]^{\frac{2}{p}}\leq \frac{1}{4}\,.
$$
Notice that in order to establish the claimed dependence of $\tilde{R}$ upon the various parameters we have used that $c_2$ depends in turn on $n,p,\ratio$.
By \rif{raggio}, as $\omega(\cdot)$ is non-decreasing, merging \rif{CCompdopo1} and \rif{sceltasmall2} yields
\begin{eqnarray}
&&\nonumber \mean{B_{R/2H}} |Du- (Du)_{B_{R/2H}}|\, dx \leq  \frac{1}{2}\mean{B_{R}} |Du- (Du)_{B_{R}}|\, dx\\
&& \hspace{4cm} + c\left[\frac{|\mu|(B_{R})}{R^{n-1}}\right]^{\frac{1}{p-1}}+ c[L_1\omega(R)]^{\frac{2}{p}}\left(|(Du)_{B_R}|+s\right)\,.\label{CCompdopo}
\end{eqnarray}
We now fix a ball $B(x_0,2R)\subseteq
 \Omega$ as in the statement of Theorem \ref{mainx}; for $i \in \{0,1,2,\ldots\}$, let us define
\eqn{defiBK}
$$
B_i:= B(x_0,R/(2H)^i):=B(x_0,R_i)\qquad \mbox{and}\qquad k_i:=|(Du)_{B_i}|\,.
$$
For every integer $m \in \en$ we write
\begin{eqnarray}
\nonumber k_{m+1}= \sum_{i=0}^m (k_{i+1}-k_i) + k_0 & \leq & \sum_{i=0}^m \mean{B_{i+1}} |Du- (Du)_{B_{i}}|\, dx + k_0\nonumber \\
&\leq &\sum_{i=0}^m (2H)^{n}\mean{B_i} |Du- (Du)_{B_{i}}|\, dx+ k_0\,.\label{CC1}
\end{eqnarray}
Therefore, defining
\eqn{deffia}
$$
A_i:=\mean{B_i} |Du- (Du)_{B_{i}}|\, dx
$$
inequality \rif{CC1} becomes
\eqn{CC3}
$$
k_{m+1}  \leq (2H)^{n}\sum_{i=0}^m A_i + k_0\,.
$$
To estimate the right hand side of the previous inequality we observe that \rif{CCompdopo} used with $R \equiv R_{i-1}$
yields, whenever $i \geq 1$
\eqn{consid}
$$ A_{i}  \leq \frac{1}{2} A_{i-1} + c\left[\frac{|\mu|(B_{i-1})}{R_{i-1}^{n-1}}\right]^{\frac{1}{p-1}}  + c\,[L_1\omega(R_{i-1})]^{\frac{2}{p}}(k_{i-1}+s)\,,
$$
with $c \equiv c (n,p,\ratio)$, where we have taken into account \rif{defiBK} and that $H$ depends on $n,p,\ratio$. We now consider \rif{consid} for $i \in\{1,\ldots,m\}$ and sum up over $i$, thereby gaining
$$ \sum_{i=1}^m A_{i}  \leq \frac{1}{2} \sum_{i=0}^{m-1}A_{i} + c\sum_{i=0}^{m-1}\left[\frac{|\mu|(B_{i})}{R_{i}^{n-1}}\right]^{\frac{1}{p-1}}  + c\sum_{i=0}^{m-1}\,[L_1\omega(R_{i})]^{\frac{2}{p}}(k_{i}+s)\,,
$$
and therefore
$$ \sum_{i=1}^m A_{i}  \leq A_{0} + 2c\sum_{i=0}^{m-1}\left[\frac{|\mu|(B_{i})}{R_{i}^{n-1}}\right]^{\frac{1}{p-1}}  + 2c\sum_{i=0}^{m-1}\,[L_1\omega(R_{i})]^{\frac{2}{p}}(k_{i}+s)\,.
$$
Using the last inequality in \rif{CC3} yields, for every integer $m \geq 1$
\eqn{qquu2}
$$
k_{m+1}\leq c\left(A_0 + k_0 +  \sum_{i=0}^{m-1} \left[\frac{|\mu|(B_i)}{R_{i}^{n-1}}\right]^{\frac{1}{p-1}}\right)  + c\sum_{i=0}^{m-1}  [L_1\omega(R_{i})]^{\frac{2}{p}}(k_i+s)\,,
$$
as usual for a constant $c \equiv c(n,p,\ratio)$.

{\em Step 2: Wolff Potentials and Dini-continuity.} We notice that
\begin{eqnarray*}
 \sum_{i=0}^{m-1} \left[\frac{|\mu|(B_i)}{R_{i}^{n-1}}\right]^{\frac{1}{p-1}}&\leq &
\sum_{i=0}^\infty\left[\frac{|\mu|(B_i)}{R_{i}^{n-1}}\right]^{\frac{1}{p-1}} \\
& \leq &\frac{2^{\frac{n-1}{p-1}}}{\log 2}\int_R^{2R}  \left[\frac{|\mu|(B(x_0,\varrho))}{\varrho^{n-1}}\right]^{\frac{1}{p-1}}\, \frac{d\varrho}{\varrho} + \sum_{i=0}^\infty\left[\frac{|\mu|(B_{i+1})}{R_{i+1}^{n-1}}\right]^{\frac{1}{p-1}}\\ & \leq &
\frac{2^{\frac{n-1}{p-1}}}{\log 2} \int_R^{2R}  \left[\frac{|\mu|(B(x_0,\varrho))}{\varrho^{n-1}}\right]^{\frac{1}{p-1}}\, \frac{d\varrho}{\varrho} \\ &&  \qquad + \frac{(2H)^{\frac{n-1}{p-1}}}{\log 2H}\sum_{i=0}^\infty\int_{R_{i+1}}^{R_i}  \left[\frac{|\mu|(B(x_0,\varrho))}{\varrho^{n-1}}\right]^{\frac{1}{p-1}}\, \frac{d\varrho}{\varrho}\\&\leq &
\left(\frac{2^{\frac{n-1}{p-1}}}{\log 2}+\frac{(2H)^{\frac{n-1}{p-1}}}{\log 2H}\right)\npma(x_0,2R)\,.
\end{eqnarray*}
Recalling the dependence of $H$ determined in \rif{sceltasmall} we conclude that there exists a constant $c$ depending only on $n,p,\ratio$ such that
\eqn{primapot}
$$
\sum_{i=0}^{m-1} \left[\frac{|\mu|(B_i)}{R_{i}^{n-1}}\right]^{\frac{1}{p-1}}\leq c \npma(x_0,2R)
$$
holds whenever $m \in \en$. As for the last term in \rif{qquu2}, using the fact that $\omega(\cdot)$ is non-decreasing, we estimate
\begin{eqnarray*}
 \sum_{i=0}^{m-1}  [\omega(R_{i})]^{\frac{2}{p}}& \leq  &
\sum_{i=0}^{\infty
}  [\omega(R_{i})]^{\frac{2}{p}} \\ &\leq &\frac{1
}{\log 2}\int_R^{2R}  [\omega(\varrho)]^{\frac{2}{p}}\, \frac{d\varrho}{\varrho} + \sum_{i=0}^\infty[\omega(R_{i+1})]^{\frac{2}{p}}\\ & \leq &
\frac{1 }{\log 2}\int_R^{2R}  [\omega(\varrho)]^{\frac{2}{p}}\, \frac{d\varrho}{\varrho}+ \frac{1
}{\log 2H}\sum_{i=0}^\infty\int_{R_{i+1}}^{R_i} [\omega(\varrho)]^{\frac{2}{p}}\, \frac{d\varrho}{\varrho}\\&\leq &
 \left(\frac{1 }{\log 2}+\frac{1}{\log 2H}\right)\int_{0}^{2R} [\omega(\varrho)]^{\frac{2}{p}}\, \frac{d\varrho}{\varrho}\,.
\end{eqnarray*}
Recalling the definition of $d(\cdot)$ in \rif{intdini}, and the fact that $H\geq 1$ we have
\eqn{secondapot}
$$
 \sum_{i=0}^{m-1} [ \omega(R_{i})]^{\frac{2}{p}}\leq \frac{2d(2R) }{\log 2}\,.
$$
For later use, we now further restrict the value of $\tilde{R}$ in order to have
\eqn{sceltadi}
$$
L_1^{\frac{2}{p}}d(2\tilde{R})\leq 1 \,.
$$
In light of inequalities \rif{primapot} and \rif{secondapot}, 
\rif{qquu2} yields, for every $m\geq 1$
\eqn{mainestx2}
$$
k_{m+1}  \leq c_3M   + c_4\sum_{i=0}^{m-1}  [L_1\omega(R_{i})]^{\frac{2}{p}}k_i\,,
$$
where $c_3, c_4$ depend only on $n,p,\ratio$ and we have set
\eqn{Msetting}
$$
M := \mean{B(x_0,R)}(|Du|+s)\, dx + \npma(x_0,2R)\,.
$$
Notice that, using \rif{sceltadi} and H\"older's inequality, to get \rif{mainestx2} from \rif{qquu2} we have also estimated
\eqn{triest}
$$
A_0 + k_0+ d(2
R)s\leq c\mean{B(x_0,R)}(|Du|+s)\, dx \,.
$$
Finally, as in \rif{CC1}, by estimating
$
k_1 = (k_1-k_0) + k_0 \leq c A_0 + k_0,
$
and using \rif{triest}, we complement \rif{mainestx2} with
\eqn{starti}
$$
k_0 + k_1 \leq c_3M\,
$$
that we obtain by enlarging $c_3$ a bit.

{\em Step 3: Induction and conclusion.} We restrict the value of $\tilde{R}$ for the last time in order to have
\eqn{sceltadi2}
$$
d(2\tilde{R})\leq \frac{1}{8c_4L_1^{2/p}}
$$
where $c_4\equiv c_4(n,p,\ratio)$ is the constant appearing in \rif{mainestx2}; this, together with the choices \rif{sceltasmall2} and \rif{sceltadi2},  finally determines $\tilde{R}$ as a positive quantity depending only on $ n, p,\ratio, L_1, \omega(\cdot)$, as required in the statement of the Theorem. We now prove that the following inequality holds whenever $m \geq 1$:
\eqn{indu}
$$
k_{m+1}\leq 2c_3M\,,
$$
where $c_3$ and $M$ have been introduced in \rif{mainestx2} and \rif{Msetting}, respectively. We prove \rif{indu} by induction; the cases $m=-1, 0$ are settled in \rif{starti}; now we assume the validity of \rif{indu} whenever $m\leq \tilde{m}$, and prove it for $\tilde{m}+1$. We have, using \rif{mainestx2}, \rif{indu}, \rif{secondapot}, and \rif{sceltadi2}
\begin{eqnarray}
\nonumber k_{\tilde{m}+2}  & \leq  & c_3M   + c_4\sum_{i=0}^{\tilde{m}}  [L_1\omega(R_{i})]^{\frac{2}{p}}k_i\\ \nonumber & \leq & c_3M   + 2c_4c_3M\sum_{i=0}^{\tilde{m}}  [L_1\omega(R_{i})]^{\frac{2}{p}} \\ & \leq  & c_3M + \frac{4c_4c_3L_1^{2/p}}{\log 2}\,d(2R)M \nonumber \\ &\leq & c_3M + 8c_4c_3L_1^{2/p}d(2R)M\nonumber\\
&\le & c_3M + c_3M = 2c_3M\,,\label{mainestx3}
\end{eqnarray}
so that \rif{indu} holds whenever $m \in \en$. We now recall that, being $Du$ assumed to be a continuous vector field, for every $x_0 \in \Omega$ it holds that
\eqn{lebe}
$$
|Du(x_0)|= \lim_{m \to \infty} k_{m+1}\leq 2c_3M\,.
$$ Merging the last inequality with \rif{Msetting} finally yields \rif{mainestx}. For the case with no $x$-dependence in the vector field $a(\cdot)$ we just notice that in this case the last term in \rif{CCompdopo0} does not appear, and the same happens in \rif{CCompdopo}, therefore no smallness assumption of the type \rif{sceltasmall2} or \rif{sceltadi2} are required, and all the estimates in the iteration procedure work with no restriction of the type \rif{raggio}. The final outcome is that estimate \rif{mainestx} holds for every ball $B(x_0,2R)\subseteq \Omega$. We finally observe that the only point in this proof where we used that $u \in C^1(\Omega)$ was at the very end, in order to assert \rif{lebe} at every point $x_0$; otherwise we could have just assumed $u \in W^{1,p}_{\loc}(\Omega)$ to get \rif{lebe} almost everywhere.
\end{proof}
\section{Parabolic estimates}\label{paraproof}
This section is devoted to the proof of Theorems \ref{mainp} and \ref{mainp2}, and therefore we shall in general argue under assumptions \rif{par1} and \rif{par2}, respectively. 
We observe that assumptions \rif{par1}$_{1}$-\rif{par1}$_{2}$ obviously imply \rif{par2}, and this fact will be implicitly used several times in the following.
\subsection{Basic estimates from Nash-Moser's theory}\label{pp1} Here we shall emphasize a few properties of solutions $\tu \in C^{0}(t_1, t_2;L^{2}(B(\bar{x},\gamma)))\cap L^{2}(t_1,t_2;W^{1,2}(B(\bar{x},\gamma)))$ for $\gamma>0$, to homogeneous , non-linear, parabolic equations of the type
\eqn{basicparbb}
$$
\tu_t - \divo \ b(x,t,D\tu)=0\,,
$$
therefore considered in the basic cylinder $\tilde{Q}\equiv B(\bar{x},\gamma)\times (t_1, t_2)$. The vector field $b \colon \tilde{Q} \times \er^n \to \er^n$ is supposed to be only Carath\`eodory regular - in particular the dependence on the coefficients $(x,t)$ is merely measurable. The next result essentially encodes Nash-Moser's regularity for parabolic equations.
\begin{prop}\label{decbase} Let $\tu \in C^{0}(t_1, t_2;L^{2}(B(\bar{x},\gamma)))\cap L^{2}(t_1,t_2;W^{1,2}(B(\bar{x},\gamma)))$ be a weak solution to \trif{basicparbb}, under the assumptions \trif{par2}. Then $u \in C^{0,\beta}_{\loc}(\tilde{Q})$ for some $\beta \in (0,1]$ depending only on $n, \nu, L$, and moreover, there exists a constant $c$, again depending only on $n, \nu, L$, such that the following inequality:
\eqn{estpbase}
$$
\mean{Q_\varrho} |\tu - (\tu)_{Q_\varrho}|\, dx \, dt \leq c \left(\frac{\varrho}{R}\right)^{\beta}\mean{Q_R} |\tu - (\tu)_{Q_R}|\, dx \, dt +c sR
$$
holds whenever $Q_\varrho \subseteq Q_R \subseteq \tilde{Q}$ are parabolic cylinders with the same vertex\,.
\end{prop}
\begin{proof} We shall rely on the precise estimates given by Lieberman in \cite[Chapter 6]{lbook}; this is in turn an extension of the classical Nash-Moser's theory for linear parabolic equations \cite{moser, nash}. We first observe that by a simple application of Young's inequality from \rif{par2} it also follows
\eqn{crescitap}
$$
(\nu/2)|z|^2 \leq \langle b(x,t,z), z\rangle + cs^2,
$$
where $c\equiv 2L^2/\nu$; compare with \rif{veramon}.
Therefore applying \cite{lbook} we have
\eqn{basicestp}
$$
\oscs_{Q_\varrho}\, \tu \leq c \left(\frac{\varrho}{R/2}\right)^{\beta}\oscs_{Q_{R/2}}\, \tu + csR\qquad \mbox{for every} \  \varrho \leq R/2\,,
$$
for constants $c\geq 1$ and $\beta \in (0,1]$, depending only on $n,\nu, L$. Here we are adopting the standard notation
$$
\oscs_{Q_\varrho} \, \tu= \sup_{Q_\varrho} \, |\tu(x,t)-\tu(\tilde x,\tilde t )|\qquad \qquad (x,t),(\tilde x,\tilde t )\in Q_\varrho\,.
$$
Estimate \rif{basicestp} maybe necessitates a few clarifications; it can be inferred from \cite[Theorem 6.28]{lbook} once the following things are taken into account. First, one has to adapt the estimates derived in \cite[Chapter 6]{lbook}, considering the growth assumption \rif{crescitap} and to compare this with \cite[(6.20)]{lbook} and \cite[(6.28)]{lbook}. Second, the estimates in \cite[Chapter 6]{lbook} are often stated for linear operators, but the linearity is actually irrelevant, as growth conditions are the only relevant thing, therefore all the estimates can be derived for general quasilinear operators; see the comments at the beginning of \cite[Section 6.5]{lbook}. Again from \cite[Theorem 6.17]{lbook} we gain
\eqn{estsp}
$$
\sup_{Q_{R/2}} \, |\tilde u|\leq c \mean{Q_R} |\tilde u|\, dx \, dt + csR\,,
$$
again for a constant $c$ depending only on $n,\nu, L$, and whenever $Q_{R}\subseteq \tilde{Q}$. With $m \in \er$ being a fixed number, we observe that $\tu - m$ still solves \rif{basicparbb}, therefore estimate \rif{estsp} gives
\eqn{estsp2}
$$
\sup_{Q_{R/2}} \, |\tu-m|\leq c \mean{Q_R} |\tu-m|\, dx \, dt + csR\,.
$$
On the other hand we have
\eqn{estsp3}
$$
\oscs_{Q_{R/2}}\, (\tu - m) \leq 2 \sup_{Q_{R/2}} \, |\tu-m|\leq c \mean{Q_R} |\tu-m|\, dx \, dt + csR\,,
$$
and, in turn applying estimate \rif{basicestp} to $\tu - m$ we find, by mean of \rif{estsp3}
$$
\oscs_{Q_\varrho}\, \tu = \oscs_{Q_{\varrho}}\, (\tu - m)  \leq c \left(\frac{\varrho}{R}\right)^{\beta}\mean{Q_R} |\tu-m|\, dx \, dt  + csR\,.
$$
Noticing that
$$
\mean{Q_\varrho} |\tu - (\tu)_{Q_\varrho}|\, dx \, dt\leq \oscs_{Q_\varrho}\, \tu
$$
we finally obtain
$$
\mean{Q_\varrho} |\tu - (\tu)_{Q_\varrho}|\, dx \, dt\leq c \left(\frac{\varrho}{R}\right)^{\beta}\mean{Q_R} |\tu-m|\, dx \, dt  + csR\,.
$$
At this stage \rif{estpbase} follows choosing $m=(\tu)_{Q_R}$\,.
\end{proof}
\subsection{Comparison estimates}\label{pp2} In the rest of Section \ref{pp2} we keep fixed a symmetric parabolic cylinder $Q(x_0, t_0,2R)
\equiv Q_{2R} \subset\Omega_T$. We start considering the unique solution
\eqn{defiw}
$$
   w\in C^0(t_0-4R^2,t_0; L^2(B(x_0,2R)))\cap L^2(t_0-4R^2,t_0
    ;W^{1,2}(B(x_0,2R)))
$$ to the following Cauchy-Dirichlet problem:
\begin{equation}\label{CD-local}
    \left\{
    \begin{array}{cc}
    w_t-\divo \, a(x,t,Dw)=0&\mbox{in $ Q_{2R}$}\\[3pt]
    w=u&\mbox{on $\partial_{\rm par} Q_{2R}$\,.}
    \end{array}
    \right.
\end{equation}
\begin{lemma}\label{comp-lem0}
Let $u\in C^0(-T,0; L^2(\Omega))\cap L^2(-T,0;W^{1,2}(\Omega ))$ be a solution to \trif{basicpar} with $\mu \in L^2(\Omega_T)$,
under the assumptions \trif{par2}, and let $w$ be defined in \trif{defiw}-\trif{CD-local}. Then there exists a constant $c=c(n, \nu)$ such that the following estimates hold:
\begin{equation}\label{comparison}
    \mean{Q_{2R}} |Du-Dw|\, dx \, dt \leq \frac{c|\mu|(Q_{2R})}{R^{N-1}}\, ,
\end{equation}
\begin{equation}\label{comparison2}
   \mean{Q_{2R}} |u-w|\, dx \, dt \leq \frac{c|\mu|(Q_{2R})}{R^{N-2}}\, .
\end{equation}
\end{lemma}
\begin{proof}

{\em Step 1: Universal estimate}. Here we assume $(x_0,t_0)=(0,0)$ and $2R =1$ - that is $Q_{2R} \equiv Q_1$ - and that $|\mu|(Q_1)= 1$, and prove that the following universal inequality holds:
\eqn{univestimate}
$$
    \int_{Q}|Du-Dw|\, dx \, dt\leq  \frac{c(n)}{\nu}\,.$$
Our starting point here will be the parabolic estimates developed in \cite{Bmany}. As described in Section 2,
we use the Steklov-averages formulation of both \rif{basicpar} and (\ref{CD-local})$_1$, i.e., for every  $t\in(-1,0)$ and $t+h < 0$ there holds
\begin{equation}\label{CD-Steklov}
    \int_{B}\Big( \partial_tu_h(\cdot ,t)\varphi +\langle [a(\cdot ,t,Du)]_h,D\varphi\rangle\Big)\, dx
    =\int_{B}\varphi\mu_h(x, t)\, dx
\end{equation}
and
\begin{equation}\label{CD-loc-Steklov}
    \int_{B}\Big( \partial_tw_h(\cdot ,t)\varphi +\langle [a(\cdot ,t,Dw)]_h,D\varphi\rangle\Big)\, dx
    =0
\end{equation}
for every $\varphi\in C_0^\infty (B)$ with compact support, and by density whenever $\varphi\in W^{1,2}_0 (B)$. The initial datum
of $w_h$ is here taken in the sense of $L^2(B)$, which means that $ w_h(\cdot ,-1)\to u (\cdot,-1)$
in $L^2(B)$ when $h\downarrow 0$, and since $u\in C^0([-1,0); L^2(B))$, this implies in particular
that
\begin{equation}\label{initial-datum}
    \lim_{h\downarrow 0}\| (u_h-w_h)(\cdot ,-1)\|_{L^2(B)}=0\,.
\end{equation}
With the notation fixed in \rif{troncamenti},
by $\Psi_k\colon\er\to\er$ we denote the following primitive
$
    \Psi_k(s):=\int_0^\tau\Phi_k(\kappa )\, d \kappa$ for $\tau\in\er$,
that for later use we compute explicitly:
$$
    \Psi_k(\tau)=
    \left\{
    \begin{array}{cc}
    \frac12 +(\tau-k-1)    &\tau\ge k+1\\[2pt]
    \frac12 (\tau-k)^2      &k<\tau<k+1\\[2pt]
    0                  &-k\leq \tau\le k\\[2pt]
    \frac12 (\tau+k)^2     &-k-1<\tau<-k\\[2pt]
    \frac12 -(\tau+k+1)   &\tau\le -k-1\,,
    \end{array}
    \right.
$$
and note that $\Psi_k\ge 0$. By testing the difference Steklov-formulations (\ref{CD-Steklov}),
(\ref{CD-loc-Steklov})
\begin{eqnarray}
    && \nonumber \int_{B}\Big( \partial_t(u_h-w_h)(x ,t)\varphi +\langle [a(\cdot ,t,Du)]_h- [a(\cdot ,t,Dw)]_h,D\varphi\rangle\Big)\, dx
    \\ &&\qquad =\int_{B}\varphi\mu_h(x, t)\, dx\, ,\label{CD-Steklov-diff}
\end{eqnarray}
with the choice
$
    \varphi (x, t):= \zeta (t)\Phi_k(u_h-w_h)(x,t)
$,
$x\in B$, where $\zeta(\cdot)$ is a smooth function, and then integrating the resulting equality on $(-1,0)$ with respect to $t$ we obtain
\begin{eqnarray}
    &&\int_{Q}\partial_t(u_h{-}w_h)\Phi_k(u_h{-}w_h)\zeta\, dx \, dt \nonumber\\
    && \qquad +\int_{Q}
    \langle [a(\cdot ,Du)]_h- [a(\cdot ,Dw)]_h,D \Phi_k(u_h{-}w_h)\rangle\zeta\, dx \, dt \nonumber\\
    && \qquad \qquad =\int_{Q}\Phi_k(u_h{-}w_h)\zeta\mu_h
    \, dx \, dt .\label{CD-Steklov-test}
\end{eqnarray}
With $\tau \in (-1,0)$ now choose $\zeta\in W^{1,\infty}(\er)$ as follows:
$$
    \zeta_{\ep}(\kappa)\equiv \zeta(\kappa)=
    \left\{
    \begin{array}{cc}
    1    &\kappa\le \tau\\[2pt]
    1-\frac{1}{\varepsilon} (t-\kappa) &\tau<\kappa<\tau+\varepsilon\\[2pt]
    0                   &\kappa>\tau+\varepsilon\,.
    \end{array}
    \right.
$$
Then, the first integral in (\ref{CD-Steklov-test}) can be rewritten in the form
\begin{align}\nonumber
    \int_{Q}\partial_t(u_h{-}w_h)&\Phi_k(u_h{-}w_h)\zeta\, dx \, dt\\
    &=
    \int_{Q}\partial_t\big[\Psi_k(u_h{-}w_h)\zeta\big]\, dx \, dt
    -
    \int_{Q}\Psi_k(u_h{-}w_h)\zeta_t\, dx \, dt \nonumber\\
    &=-\int_{B}\Psi_k(u_h{-}w_h)(x ,-1)\, dx
    -
    \int_{Q}\Psi_k(u_h{-}w_h)\zeta_t\, dx \, dt\,.\label{first}
\end{align}
Taking into account the special choice of $\zeta$ we see that the second integral appearing
on the right-hand side of (\ref{first}) converges to $\int_{B}\Psi_k(u_h{-}w_h)(x ,\tau)\, dx$
when $\varepsilon\downarrow 0$ for almost every $\tau \in (-1,0)$, while the first integral
vanishes as $h\downarrow 0$ in light of (\ref{initial-datum}). As we have seen above the initial datum
is taken in the sense of $L^2(B)$. Combining (\ref{first}) with (\ref{CD-Steklov-test})
and letting first $\varepsilon \downarrow 0$
and then $h\downarrow 0$ in \rif{first} we obtain
\begin{align*}\nonumber
    \int_{B}&\Psi_k(u{-}w)(x ,\tau)\, dx\\
    &+\int_{-1}^{\tau}\int_{B}
    \langle a(x,t ,Du){-} a(x,t ,Dw),D \Phi_k(u{-}w)\rangle\, dx \, dt
    =\int_{-1}^{\tau}\int_{B}\Phi_k(u{-}w)\mu\, dx \, dt
\end{align*}
for almost every every choice of $\tau\in (-1,0)$. This leads to
\begin{align}\nonumber
    \sup_{-1<\tau<0}\int_{B}&\Psi_k(u{-}w)(x ,\tau)\, dx\\ \nonumber
    &+\int_{Q}
    \langle a(x,t ,Du){-} a(x,t ,Dw),D \Phi_k(u{-}w)\rangle\, dx \, dt\\ & \qquad \qquad
    \le\int_{Q}|\Phi_k(u{-}w)||\mu|\, dx \, dt\, .\label{start}
\end{align}
We first exploit the properties of $\Psi_k$ in order to estimate the first term of the
left-hand side from below. We start with the obvious estimate
\eqn{e1}
$$
    \sup_{-1<\tau<0}\int_{B}\Psi_0(u{-}w)(\cdot ,\tau)\, dx\leq |\mu|(Q)\le 1\, ,
$$
which follows by applying (\ref{start}) for $k=0$; note also that $|\mu|(Q)= 1$ by assumption.
With $\tau$ being fixed, decomposing $$B=\{ x\in B: |(u{-}w)(x ,\tau)|\le 1\}\cup\{ x\in B: |(u{-}w)(x,\tau)|> 1\}=: B'_{\tau} \cup B''_{\tau}$$ we have for almost every $\tau\in (-1,0)$ that
\begin{align}
    \nonumber \int_{B}&\Psi_0(u{-}w)(x ,\tau)\, dx\\
    &=
    \int_{B'_{\tau}} {\textstyle \frac12} |(u{-}w)(x ,\tau)|^2\, dx
    +
    \int_{B''_{\tau}}   |(u{-}w)(x ,\tau)|\, dx- {\textstyle \frac12}|B''_{\tau}|
    \label{e2}
\end{align}
holds. On the other hand we trivially have
$
    \int_{B'_{\tau}}|(u{-}w)(x ,\tau)|\, dx\le |B|\,,
$
so that merging the last estimate with \rif{e1}-\rif{e2} we have
$$
    \int_{B}|(u{-}w)(x ,\tau)|\, dx\le {\textstyle \frac12}|B|+|\mu|(Q)\le c(n)
$$
which holds again for almost every $\tau \in (-1,0)$, and therefore we conclude with the $L^\infty{-}L^1$-bound
\begin{equation}\label{L-inf-L1}
    \|u-w\|_{L^\infty (-1,0;L^1(B))}:=\sup_{-1<\tau<0}\int_B|(u-w)(x,\tau)|\, dx
    \le c(n)\, .
\end{equation}
We now start exploiting the second integral appearing on the left-hand side
of (\ref{start}). For this we define
$$
    A_k:=\{ (x,t)\in Q:k\le |u(x,t)-w(x,t)|<k+1\}\,.
$$
Then, using the fact that $D\Phi_k(u-w)=Du-Dw$ on $A_k$ and $D\Phi_k(u-w)=0$ otherwise together
with \rif{par2}$_2$, $\Phi_k(\cdot)\le 1$ and $|\mu|(Q)= 1$ in (\ref{start}) we obtain
$$
     \nu \int_{A_k}|Du-Dw|^2\, dx \, dt\le
    \int_{A_k}
    \langle a(x,t ,Du){-} a(x,t ,Dw),D(u{-}w)\rangle\, dx \, dt
    \le 1
$$
so that the estimate
\begin{equation}\label{est-on-Ak}
    \int_{A_k}|Du-Dw|^2\, dx \, dt\le\frac{1}{\nu}
\end{equation}
holds for every $k\in\{0,1,2,\ldots\}$. Now, for $\lambda >1$ we estimate
\begin{align}\nonumber
    \int_{Q}\frac{|Du-Dw|^2}{(1+|u-w|)^\lambda}\, dx \, dt
    &=\sum_{k=0}^\infty \int_{A_k}\frac{|Du-Dw|^2}{(1+|u-w|)^\lambda}\, dx \, dt\\
    &\le \sum_{k=0}^\infty\frac{1}{(1+k)^\lambda}\int_{A_k}|Du-Dw|^2\, dx \, dt\le \frac{\lambda }{\nu (\lambda -1)}\, .\label{prev}
\end{align}
Here we have used in the last line the uniform estimate (\ref{est-on-Ak}). This allows us to argue
as follows: For almost every $\tau\in (-1,0)$ we have by the Cauchy-Schwartz inequality that
\begin{align*}
    \int_B|D(u&-w)(x ,\tau)|\, dx
    =\int_B\frac{|D(u-w)(x ,\tau)|}{(1+|(u-w)(x ,\tau)|)^{\lambda /2}}(1+|(u-w)(x,\tau)|)^{\lambda /2}\, dx\\
    &\le \bigg(\int_B\frac{|D(u-w)(x ,\tau)|^2}{(1+|(u-w)(x ,\tau)|)^{\lambda}}\, dx\bigg)^\frac{1}{2}
    \bigg(\int_B(c(n)+|(u-w)(x ,\tau)|)^{\lambda }\, dx\bigg)^\frac{1}{2}\,.
\end{align*}
Integrating the preceding inequality with respect to $\tau$  on $(-1,0)$ and using Cauchy-Schwartz with respect to
$\tau$ together with (\ref{prev}) then yields
\begin{align}\nonumber
    \int_{Q}|Du&{-}Dw|\, dx \, dt\\
    &\le \int_{-1}^0\bigg(\int_B\frac{|D(u{-}w)(x ,t)|^2}{(1+|(u{-}w)(x ,t)|)^{\lambda}}\, dx\bigg)^\frac{1}{2}
    \bigg(\int_B(1+|(u{-}w)(x ,t)|)^{\lambda }\, dx\bigg)^\frac{1}{2}\, dt\nonumber\\
    &\le \bigg(\int_Q\frac{|D(u{-}w)|^2}{(1+|u{-}w|)^{\lambda}}\, dx \, dt\bigg)^\frac{1}{2}
       \bigg(\int_{-1}^0\int_B(1+|(u{-}w)(x ,t)|)^{\lambda }\, dx\, dt\bigg)^\frac{1}{2} \nonumber\\
    &\le \sqrt{\frac{2^{\lambda {-}1}\lambda}{\nu (\lambda {-}1)}}~\bigg[1+\bigg(
    \int_{-1}^0\|(u{-}w)(x ,t)\|_{L^\lambda (B)}^\lambda\, dt\bigg)^\frac{1}{2}
    \bigg]\, .\label{prev1}
\end{align}
At this stage we use
a well-known version of the Gagliardo-Nirenberg embedding
theorem which reads in our setting as follows:
$$
    \|(u{-}w)(x ,t)\|_{L^\lambda (B)}\le c(n)\|D(u{-}w)(x ,t)\|_{L^1 (B)}^\theta
    \|(u{-}w)(x ,t)\|_{L^1 (B)}^{1-\theta}\, ,
$$
which holds for every choice of $\theta$ and $\lambda$ satisfying
$$
    0\le\theta\le 1,\qquad \qquad 1<\lambda <\infty ,\qquad \qquad \frac{1}{\lambda}=1-\frac{\theta}{n}\,.
$$
We note that we have $\lambda>1 $ at our disposal to ensure $\theta \in (0,1)$. Recalling the
$L^\infty{-}L^1$-estimate (\ref{L-inf-L1}) the second integral on the right-hand side of
the preceding inequality is bounded by $c(n)^{1-\theta}$, and therefore
$$
    \int_{-1}^0\|(u{-}w)(x ,t)\|_{L^\lambda (B)}^\lambda\, dt
    \le c(n)^{1+\lambda (1-\theta )}\int_{-1}^0\|D(u{-}w)(x ,t)\|_{L^1 (B)}^{\lambda\theta}\, dt\, .
$$
We now perform the choice of $\theta$ and $\lambda$. In order to have $\theta\lambda =1$, i.e.\
$\theta =\frac{1}{\lambda}$, the identity relating $\theta$ and $\lambda$ yields $\lambda =\frac{n+1}{n}$,
so that $\theta =\frac{n}{n+1}$. This implies that
$$
    \int_{-1}^0\|(u{-}w)(x ,t)\|_{L^{1+\frac{1}{n}} (B)}^{1+\frac{1}{n}}\, dt
    \le c(n)\int_{Q}|D(u{-}w)|\, dx \, dt\, .
$$
Inserting this last estimate in (\ref{prev1}), and using Young's inequality we arrive at the following inequality, which is turn implies  \rif{univestimate}
\begin{align*}
    \int_{Q}|Du-Dw|\, dx \, dt
    &\le \frac{c(n)}{\sqrt{\nu}} \bigg[ 1+\bigg(\int_{Q}|Du-Dw|\, dx \, dt\bigg)^\frac{1}{2}\bigg]\\
    &\le  \frac{c(n)}{\nu} +\frac12  \int_{Q}|Du-Dw|\, dx \, dt\, .
\end{align*}

{\em Step 2: General case and scaling}. Similarly to the elliptic case we first reduce to the case $Q_{2R}= Q(x_0,t_0;2R)\equiv Q$ by changing variables and passing to new solutions and vector fields. More precisely, for $(\tilde x, \tilde t)\in Q$
\eqn{change}
$$
    \left\{
    \begin{array}{c}
    \displaystyle \tilde u(\tilde x, \tilde t):=\frac{u(x_0+2R\tilde x,t_0+4R^2\tilde t)}{2R} \\[8 pt]  \displaystyle \tilde w(\tilde x, \tilde t):=\frac{w(x_0+2R\tilde x,t_0+4R^2\tilde t)}{2R} \\[8 pt]
    \tilde a(\tilde x, \tilde t,z):=a(x_0+2R\tilde x,t_0+4R^2\tilde t,z)\\ [3pt] \tilde \mu(\tilde x, \tilde t):=2R\mu(x_0+2R\tilde x,t_0+4R^2\tilde t).
    \end{array}
    \right.
$$
Then it follows that $\tilde u_{\tilde t} -\divo\,  \tilde a(\tilde x, \tilde t,D\tilde u)=\tilde \mu$
and $\tilde w_{\tilde t} -\divo\,  \tilde a(\tilde x, \tilde t,D\tilde w)= 0$ in $Q$, and $\tilde u=\tilde w$
on $\partial_{\rm par}Q$. Furthermore the new vector field
$\tilde a(\cdot)$ satisfies \rif{par2}.
To remove the additional
assumption $|\mu|(Q)=1$ we assume, as we can,
that $A:= |\mu|(Q)>0$ - otherwise there is nothing to prove since the strict monotonicity of the vector field would imply $u\equiv w$ - and then we re-scale $u$, $w$, $a(\cdot)$ and $\mu$ as follows:
$
    \bar u:=A^{-1}u$, $ \bar w:=A^{-1}w$, $\bar a(x,t,z):= A^{-1}a(x,t,Az)$, $\bar \mu:= A^{-1}\mu$.
Then, it is easily verified that $\bar u$ is a solution of $\bar u_t- \divo\, \bar a(x,t,\bar u)=\bar \mu$
on $Q$, that $\bar w$ solves $\bar w_t- \divo\, \bar a(x,t,\bar w)=0$
on $Q$, that $\bar u=\bar w$ on $\partial_{\rm par}Q$ and $ |\mu|(Q)=1$.
Moreover, $\bar a(\cdot)$ fulfills the
strict monotonicity condition \rif{par2} which is the only one used in a quantitative way
in the derivation of the universal comparison estimate \rif{univestimate}. Therefore estimate \rif{univestimate} applied to $\bar u- \bar{w}$ yields
$$\nu\int_{Q} |D\bar u-D\bar w|\, dx \, dt\le c(n)\,,$$ and re-scaling back this inequality from $\bar u-\bar w$
to $u-w$ leads us to the comparison estimate (\ref{comparison}) with $R=1$, and by the previous scaling \rif{comparison} is proved in the general case. In order to get \rif{comparison2} it is sufficient to recall that $(u-w)(\cdot, t) \in W^{1,2}_0(B(x_0,2R))$ for almost every $t \in (t_0-4R^2, t_0)$ and therefore applying Poincar\'e's inequality slicewise and integrating we gain
$$
\int_{Q(x_0,t_0;2R)} |u-w|\, dx\, dt  \leq c R \int_{Q(x_0,t_0;2R)} |Du-Dw|\, dx \, dt
$$
which together with \rif{comparison} implies \rif{comparison2}. This completes the proof.
\end{proof}
We now proceed with a further comparison estimate; after having defined the comparison solution $w$ in \rif{defiw} we define the unique solution
\eqn{defiv}
$$
   v\in C^0(t_0-R^2,t_0; L^2(B(x_0,R)))\cap L^2(t_0-R^2,t_0
    ;W^{1,2}(B(x_0,R)))
$$ of the following Cauchy-Dirichlet problem:
\begin{equation}\label{CD-localv}
    \left\{
    \begin{array}{cc}
    v_t-\divo \, a(x_0,t,Dv)=0&\mbox{in $ Q_{R}$}\\[3pt]
    v=w&\mbox{on $\partial_{\rm par} Q_{R}$\,.}
    \end{array}
    \right.
\end{equation}
We recall the reader's attention on the fact that in \rif{CD-localv} we have frozen the coefficients only with respect to the space variable.
\begin{lemma}\label{comp-lem}
Let $u\in C^0(-T,0; L^2(\Omega))\cap L^2(-T,0;W^{1,2}(\Omega))$ be a solution to \trif{basicpar} with $\mu \in L^1(\Omega_T)$,
under the assumptions \trif{par2}, let $w$ be defined in \trif{defiw}-\trif{CD-local}, and finally let $v$ be defined in \trif{defiv}-\trif{CD-localv}. Then there exists a constant $c\equiv c(n, \nu, L)$ such that the following estimate holds:
\begin{eqnarray}\label{comparisonpar}
    \nonumber\mean{Q_{R}} |Du-Dv|\, dx \, dt & \leq & c\left[1+L_1\omega(R)\right]\frac{|\mu|(Q_{2R})}{R^{N-1}}\\ &&\qquad + cL_1\omega(R)\mean{Q_{2R}} (|Du|+s)\, dx \, dt \, .
\end{eqnarray}
\end{lemma}
\begin{proof} Let us first get the estimate
\begin{equation}\label{comparisonpar0}
     \mean{Q_{R}} |Dv-Dw|^2\, dx \, dt \leq cL_1^2[\omega(R)]^2\mean{Q_{R}} (|Dw|+s)^2\, dx \, dt \, .
\end{equation}
The proof consists of a rather standard comparison argument which we report here for the sake of completeness. The following computations are formal, as they would need the existence of the time derivatives for both $v$ and $w$; on the other hand they can be made rigorous by using the Steklov-averages formulation exactly as done for \rif{CD-Steklov-test}-\rif{first}. Using the fact that both $v$ and $w$ are solutions yields
\begin{eqnarray}
&& \nonumber(v-w)_t -\divo \, \left(a(x_0,t,Dv)-a(x_0,t,Dw)\right)\\
&& \qquad \qquad = \divo \, \left(a(x_0,t,Dw)-a(x,t,Dw)\right)\,.\label{diffeq}
\end{eqnarray}
By testing the weak form of the previous equality with $v-w$ - here we need to pass to Steklov averages - we get, after standard parabolic manipulations and the fact that $v$ and $w$ agree on $\partial_{\rm par} Q_{R}$, that
\begin{eqnarray*}
&& \sup_{t_0-R^2< t < t_0} \int_{B(x_0,R)} |v-w|^2(x,t)\, dx \\ && \qquad \qquad  + \int_{Q_{R}} \langle a(x_0,t,Dv)-a(x_0,t,Dw), Dv- Dw \rangle \, dx\, dt\\
&& \qquad \qquad \qquad \leq
\left|
\int_{Q_{R}} \langle a(x_0,t,Dw)-a(x,t,Dw), Dv- Dw \rangle \, dx\, dt\right|\,.
\end{eqnarray*}
Discarding the first term in the previous inequality, using \rif{par2} and \rif{par1}$_3$, yields
$$
\mean{Q_R} |Dv-Dw|^2\, dx\, dt  \leq cL_1\omega(R)\mean{Q_R} (|Dw|+s)|Dv-Dw|\, dx\, dt\,.
$$
At this stage \rif{comparisonpar0} follows by the last estimate and a standard use of Young's inequality. The next step is to recall a higher integrability estimate which follows form the application of Gehring's lemma in the parabolic setting. The statement can be found in several paperssee for instance \cite{NW}, once taking into account the growth conditions \rif{par1} considered here. There exist constants $c\geq 1$ and $\chi>2$, depending only on $n,\nu, L$, such that the following inequality:
$$
\left(\mean{Q_\varrho} |Dw|^{\chi} \,
dx\, dt \right)^{\frac{1}{\chi}} \leq c  \left(\mean{Q_{2\varrho}}
(|Dw|+s)^{2} \, dx\, dt\right)^{\frac{1}{2}} $$
holds whenever $Q_{2\varrho} \subset Q_{2R}$. We are therefore in position to apply Lemma \ref{revq},
 which allows to establish
\eqn{revpar}
$$
\left(\mean{Q_R} |Dw|^{2} \,
dx\, dt \right)^{\frac{1}{2}} \leq c \mean{Q_{2R}}
(|Dw|+s) \, dx\, dt\,.$$
Combining \rif{comparisonpar0} and \rif{revpar} via H\"older's inequality we gain
$$
\mean{Q_{R}} |Dv-Dw|\, dx \, dt \leq L_1\omega(R) \mean{Q_{2R}}
(|Dw|+s) \, dx\, dt\,.
$$
Again combining the last estimate with \rif{comparison}, and actually using it twice, finally gives \rif{comparisonpar}.
\end{proof}
We conclude this section with yet another a priori estimate, in which, {\em for the first and only time} in this section we use the complete assumptions \rif{par1} rather than the weaker ones in \rif{par2}.
\begin{prop} Let $v$ be defined in \trif{defiv}-\trif{CD-localv}; there exist constants $c \geq 1$ and $\beta \in (0,1]$, depending only on $n,\nu,L$, such that the following inequality:
\eqn{estpbasevv}
$$
\mean{Q_\varrho} |D v - (D v)_{Q_\varrho}|\, dx \, dt \leq c \left(\frac{\varrho}{R}\right)^{\beta}\mean{Q_R} |D v - (D v)_{Q_R}|\, dx \, dt
$$
holds whenever $Q_\varrho \subseteq Q_R $
 is a parabolic cylinder with the same vertex of $Q_R$. More in general, the inequality
\eqn{compo}
$$
\mean{Q_\varrho} |D_{\xi} v - (D_{\xi} v)_{Q_\varrho}|\, dx \, dt
\leq  c \left(\frac{\varrho}{R}\right)^{\beta}\mean{Q_{R}} |D_{\xi} v - (D_{\xi} v)_{Q_{R}}|\, dx \, dt
$$
holds whenever $\xi \in \{1,\ldots, n\}$ and $\varrho \leq R$.
\end{prop}
\begin{proof} We recall the difference quotients argument developed for instance in \cite[Chapter 8]{D}. We indeed have that $v$ is higher differentiable in a smaller cylinder: $v \in L^2 (t_0-R^2/4, t_0; W^{2,2}(B(x_0,R/2)))$
 and moreover, whenever $\xi \in \{1,\ldots, n\}$ we have $D_\xi v \in  C^0 (t_0-R^2/4, t_0; L^{2}(B(x_0,R/2)))$. Finally, the function $D_\xi v$ solves the following differentiated equation:
$$
(D_{\xi} v)_t - \divo\, (\tilde{a}(x,t)DD_\xi v)=0\,,
$$
weakly in $Q(x_0,t_0;R/2)\equiv Q_{R/2}$, where the matrix $\tilde{a}_{ij}(x,t):= D_{z_j}a_i(x_0,t,Dv(x,t))$ has measurable entries. By assumptions \rif{par1} we have that the following monotonicity and growth conditions are satisfied:
$$
\nu |z|^2\leq \langle \tilde{a}(x,t)z, z \rangle\,, \qquad \qquad  |\tilde{a}(x,t)z|\leq L|z|
$$
whenever $z \in \er^n$, and $(x,t)\in Q_{R/2}$. Therefore, if we let $b(x,t,z):= \tilde{a}(x,t)z$, we may apply Proposition \ref{decbase} to $\tu \equiv D_\xi v$ in the cylinder $Q_{R/2}$, thereby obtaining
\begin{eqnarray}
\mean{Q_\varrho} |D_{\xi} v - (D_{\xi} v)_{Q_\varrho}|\, dx \, dt & \leq & c \left(\frac{\varrho}{R/2}\right)^{\beta}\mean{Q_{R/2}} |D_{\xi} v - (D_{\xi} v)_{Q_{R/2}}|\, dx \, dt\nonumber \\
&\leq & c \left(\frac{\varrho}{R}\right)^{\beta}\mean{Q_{R}} |D_{\xi} v - (D_{\xi} v)_{Q_{R}}|\, dx \, dt\label{compo0}
\end{eqnarray}
that holds whenever $\varrho \leq R/2$, and
where again $c \equiv c (n, \nu, L)$; we notice that the last estimate in \rif{compo0} has been obtained arguing as for \rif{stimalike}. Finally, the same inequality follows for $\varrho \in (R/2,R]$, trivially, and therefore \rif{compo} is completely established. In turn, since $\xi \in \{1,\ldots, n\}$ is arbitrary \rif{compo} implies \rif{estpbasevv} and the proof is complete.
\end{proof}
\subsection{Proof of Theorems \ref{mainp}-\ref{mainpa}} 
We start by Theorem \ref{mainp}.
In Sections \ref{pp1}-\ref{pp2} we have built the necessary set up in order to adapt the elliptic proof of Section \ref{ellell} to the parabolic case, and we shall therefore outline the relevant modifications.
Combining estimates \rif{comparisonpar} and \rif{estpbasevv} as done in Lemma \ref{compxxx} with estimates \rif{sottosti} and \rif{comp12}, we have the analogue
 of estimate \rif{CComp}
\begin{eqnarray}
&&\nonumber \mean{Q_\varrho} |D u- (D u)_{Q_{\varrho}}|\, dx\, dt \leq  c_1 \left(\frac{\varrho}{R}\right)^\beta \mean{Q_{2R}} |D u- (D u)_{Q_{2R}}|\, dx\, dt\\
&& \qquad  + c\left(\frac{R}{\varrho}\right)^{N}\frac{|\mu|(Q_{2R})}{R^{N-1}}+ c_2\left(\frac{R}{\varrho}\right)^{N}[L_1\omega(R)]\mean{Q_{2R}}(|D u|+s)\, dx\, dt\,,\label{CCompar}
\end{eqnarray}
where the constants $c, c_1, c_2\geq 1$ and $\beta \in (0,1]$ depend only on $n, \nu, L$; the last inequality holds whenever $Q_\varrho \subseteq Q_R$ are backward parabolic cylinders sharing the same vertex, and with $R< R_1$ where $R_1 \equiv R_1 (L_1,\omega(\cdot))>0$ being a suitable small radius; the restriction on $R$ is not necessary when $a(\cdot)$ is independent of $x$, although $a(\cdot)$ can be still depending on the time variable $t$. By \rif{CCompar}, choosing $H \in \en$ - similarly to \rif{sceltasmall} - large enough to have
$
c_1/H^\beta \leq 1/4,
$
and taking $R\leq \tilde{R}\leq R_1$ such that $c_2(2H)^{N}[L_1\omega(R)]\leq 1/4$ similarly to \rif{sceltasmall2}, and proceeding as for \rif{CCompdopo0}-\rif{CCompdopo} we obtain
\begin{eqnarray}
&&\nonumber \mean{Q_{R/H}} |D u- (D u)_{Q_{R/H}}|\, dx\, dt \leq  \frac{1}{2} \mean{Q_{2R}} |D u- (D u)_{Q_{2R}}|\, dx\, dt\\
&& \hspace{5cm} + \frac{c|\mu|(Q_{2R})}{R^{N-1}}+ cL_1\omega(R)\left[|(Du)_{Q_{2R}}|+s\right]\,.\label{CCompar2}
\end{eqnarray}
The last inequality is again valid for cylinders with same vertex and for a constant $c$ depending on $n,\nu,L$, with the restriction $R\leq \tilde{R}$.
We now just have to follow the scheme after \rif{CCompdopo}. Specifically, we let
$$
Q_i:= Q(x_0,t_0;R/(2H)^i)\,,\qquad k_i:=|(D u)_{Q_i}|\,;$$ letting - as in \rif{deffia} -
$$
A_i:= \mean{Q_i} |D u- (D u)_{Q_{i}}|\, dx\, dt
$$
we apply \rif{CCompar2} with $Q_R\equiv Q_i$, and iterating as after \rif{CC3}, we finally achieve estimate \rif{parest1}.

As for Theorem \ref{mainpa}, to obtain \rif{parest1cc} we argue as follows. We start by \rif{compo}; this holds directly for $w$ defined in \trif{defiw}-\trif{CD-local}, recall that in this case $v\equiv w$ as we do not have $x$-dependence and we do not have to freeze at $x_0$. Next we combine \rif{compo} directly with \rif{comparison} and obtain the following analogue of \rif{CCompar}:
\begin{eqnarray}
&&\nonumber \mean{Q_\varrho} |D_\xi u- (D_\xi u)_{Q_{\varrho}}|\, dx\, dt \leq  c_1 \left(\frac{\varrho}{R}\right)^\beta \mean{Q_{2R}} |D_\xi u- (D_\xi u)_{Q_{2R}}|\, dx\, dt\\
&& \hspace{6cm}  + c\left(\frac{R}{\varrho}\right)^{N}\frac{|\mu|(Q_{2R})}{R^{N-1}}\,.\label{ana1}
\end{eqnarray}
Now \rif{parest1cc} follows by iterating as in the proof of
 Theorem \ref{mainp}, by setting this time $$
Q_i:= Q(x_0,t_0;R/(2H)^i)\,,\qquad k_i:=|(D_\xi u)_{Q_i}|$$ and
$$
A_i:= \mean{Q_i} |D_\xi u- (D_\xi u)_{Q_{i}}|\, dx\, dt\,.
$$
Worth remarking differences are that in \rif{Msetting} the value of it is now
$$
M := \mean{Q(x_0,t_0;R)}(|D_\xi u|+s)\, dx\, dt + {\bf I}_{1}^\mu(x_0,t_0;2R)
$$ and this comes from the fact that, being now $\omega(\cdot)$ we also have $d(2R)=0$ and therefore \rif{triest} can be replaced by
$$
A_0 + k_0\leq c\mean{Q(x_0,t_0;R)}(|D_\xi u|+s)\, dx\, dt .
$$
\subsection{Proof of Theorem \ref{mainp2}} The main difference with respect to the proof of Theorem \ref{mainp} is that we just need to use estimate \rif{comparison2}, and no other intermediate comparison inequality of the type \rif{comparisonpar}. This is essentially the reason why Theorem \ref{mainp2} applies - as expected for zero-order estimates - to operators
 with measurable coefficients. As for the proof, we first obtain the estimate
\begin{eqnarray}\nonumber \mean{Q_\varrho} | u- (u)_{Q_{\varrho}}|\, dx\, dt & \leq & c_1 \left(\frac{\varrho}{R}\right)^\beta \mean{Q_{2R}} |u- ( u)_{Q_{2R}}|\, dx\, dt\\ && \qquad  +c\left(\frac{R}{\varrho}\right)^{N}\frac{|\mu|(Q_{2R})}{R^{N-2}}+ cRs\label{CCompar3}
\end{eqnarray}
which is the analogue of \rif{ana1} and of \rif{CCompar2} when no space coefficients come into the play. As usual, \rif{CCompar3} works for backward cylinders with the same vertex, and the constants $c, c_1\geq 1$ and $\beta \in (0,1]$ depending
 only on $n, \nu, L$. Estimate \rif{CCompar3} can be derived by first applying Proposition \ref{decbase} with $\tu \equiv w$ and $b(\cdot)\equiv a(\cdot)$, thereby getting the reference estimate
\eqn{CCompar33}
$$
\int_{Q_\varrho} |w - (w)_{Q_\varrho}|\, dx \, dt \leq c \left(\frac{\varrho}{R}\right)^{N+\beta}\int_{Q_{2R}} |w - (w)_{Q_{2R}}|\, dx \, dt + cRs\,,
$$
whenever $Q_{\varrho}\subseteq Q_{2R}$ is a cylinder with the vertex of $Q_{2R}$, and where $c\geq 1$ and $\beta \in (0,1]$ depend only on $n, \nu, L$. Combining this last estimate with \rif{comparison2} in the same way of Lemma \ref{compxxx} - and exactly for \rif{CCompar} - finally yields \rif{CCompar3}. By \rif{CCompar3}, choosing $H$ to have
$
c_1/H^\beta \leq 1/2,
$
we obtain, again for $c \equiv c(n, \nu, L)$
\eqn{cccccc}
$$ \mean{Q_{R/H}} | u- ( u)_{Q_{R/H}}|\, dx\, dt \leq  \frac{1}{2} \mean{Q_{2R}} |u- (u)_{Q_{2R}}|\, dx\, dt + \frac{c|\mu|(Q_{2R})}{R^{N-2}}+ cRs\,.
$$
The last estimate can be iterated as done for \rif{CCompar2}, as after \rif{CCompdopo}; we let
$$
Q_i:= Q(x_0,t_0;R/(2H)^i)\,, \qquad k_i:=|(u)_{Q_i}|$$ and finally
$$
A_i:= \mean{Q_i} | u- ( u)_{Q_{i}}|\, dx\, dt
$$
and apply \rif{cccccc} with $Q_R \equiv Q_i$, iterating as after \rif{CC3}, thereby concluding the proof.
\begin{remark}[Zero order elliptic estimate]\label{alternative} The method of proof used for Theorem \ref{mainp2} easily leads to an alternative proof of the zero order estimate \rif{KM} - even with $\gamma=1$ - from \cite{KM, TW} for the elliptic case \rif{baseq} under the assumptions
\begin{equation}\label{parell2}
    \left\{
    \begin{array}{c}
 |a(x,z)|\leq L(|z|+s)^{p-1} \\[4pt]
    \nu|z_2-z_1|^p \leq \langle a(x,z_2)-a(x,z_1), z_2-z_1\rangle
    \end{array}
    \right.
\end{equation}
whenever $z_1,z_2 \in \er^n$ and $x \in \Omega$; compare with \rif{par2}. We observe that this proof works also in the case of general signed measures, and the final outcome is
\eqn{KM222}
$$
|u(x_0)| \leq c  \mean{B(x_0,R)}(|u|+Rs)\, dx  + c\ww(x_0,2R)\,.
$$
The proof in question seems to be slightly easier that those proposed up to now, although the assumptions \rif{parell2} are slightly stronger than those in \cite{TW}. For the proof, following the scheme of the proof of Theorem \ref{mainp2}, we essentially need to have two ingredients: a reference decay estimate for comparison function $w$ introduced in \rif{Dirc1}, and a suitable comparison estimate between $u$ and $w$. This last one has already been proved in \rif{comp111}, while following for instance \cite[Theorem 5.1]{mis2}, and using as a starting point the proof of the decay estimate in \cite[Theorem 7.7]{G}, we have that the following estimate holds whenever $B_\varrho \subset B_{2R}$ is a concentric to $B_{2R}$ for the function $w$ in \rif{Dirc1}:
$$
\mean{B_\varrho} |w - (w)_{B_\varrho}|\, dx  \leq c \left(\frac{\varrho}{R}\right)^{\beta}\mean{B_{2R}} |w - (w)_{B_{2R}}|\, dx +cRs\,.
$$
At this stage combining the previous estimate with \rif{comp111} in the same way as we combined \rif{CCompar33} with \rif{comparison2} in the proof of Theorem \ref{mainp2}, and iterating as in the proof of Theorem \ref{mainx}, we finally come up with \rif{KM222}. We just mention that, with the notation introduced in the proof of Theorem \ref{mainx}, here we have to define $$k_i:=|(u)_{B_i}|\,,\qquad \mbox{and}\qquad
A_i:= \mean{B_i} |u- (u)_{B_{i}}|\, dx
\,,$$ while no assumption other than measurability is to be assumed on the partial map $x \mapsto a(x, \cdot)$.
\end{remark}
\section{General weak, and very weak, solutions}\label{app}
\subsection{General elliptic problems} Here we consider Dirichlet problems of the type \rif{Dir1} assuming that $\mu \in \MM(\Omega)$ is a general Radon measure with finite total mass, and prove that estimate \rif{mainestx} holds for general weak and very weak solutions. We are hereby considering the case of zero boundary datum for simplicity, but more general cases can be considered; see also Theorem \ref{mainx22} below for general weak (energy) solutions. We need some preliminary terminology. As mentioned in Section 2 uniqueness of general very weak solutions fails, therefore one is led to consider very weak solutions enjoying additional properties, which in some situations are unique; one of such classes, of interest here, is the one of {\bf Solutions Obtained by Limit of Approximations}  (\textnormal{SOLA}). For the sake of completeness we here recall the approximation procedure in question; the main reference here are the works of Boccardo \& Gall\"ouet \cite{BG1, BG2} and Dall'Aglio \cite{Dal}. We consider a standard, symmetric and non-negative mollifier $\phi
\in C^{\infty}_0(B_1)$ such that $\|\phi\|_{L^1(\er^n)}=1$, and then
define, for every positive integer $h$, the mollifier $\phi_h(x):=h^n\phi(hx)$. Finally the
functions $\mu_h:\er^n \to \er$ are defined via convolution, $
\mu_h(x):=(\mu*\phi_h)(x).$ Next, by standard monotonicity methods, we find a unique solution $u_h \in W^{1,p}_0(\Omega)$ to \eqn{Dirapp}
$$
\left\{
    \begin{array}{cc}
    -\divo \ a(x,Du_h)=\mu_h & \qquad \mbox{in $\Omega$}\\
        u_h= 0&\qquad \mbox{on $\partial\Omega$.}
\end{array}\right.
$$
Up to passing to a not relabeled subsequence we may assume that $ \mu_h \rightharpoonup \mu
$ weakly in the sense of measures, while the results in \cite{BG1, BG2} imply \eqn{convergenza}
$$ Du_h \to Du \quad \mbox{ strongly in}\
L^{q}(\Omega)\  \mbox{for every}\ \ q < \frac{n(p-1)}{n-1}\,,\ \ \mbox{and a.e.}$$ so that \rif{Dir1} is solved by $u$ in the usual distributional sense, and therefore $u$ is a SOLA to \rif{Dir1}; this of being a limit solution of more regular solutions indeed defines SOLA. Moreover, by \cite{boccardo, Dal}, in the case $\mu \in L^1(\Omega)$ we also have that $u$ is the only SOLA of \rif{Dir1}, in the sense that if $v\in W^{1,p-1}_0(\Omega)$ is a distributional solution to \rif{Dir1}$_1$ obtainable as a pointwise limit of solutions $v_h\in W^{1,p}_0(\Omega)$ to problems of the type \rif{Dirapp} with $\mu_h$ replaced by $\tilde{\mu}_h$ and $\tilde{\mu}_h \rightharpoonup \mu$ weakly in $L^1(\Omega)$, then we have $u\equiv v$; for more information on uniqueness see Remark \ref{uniche} below.
Now we turn back to the proof of Theorem \ref{mainx}, and in particular to estimate \rif{qquu2}, that we write when applied to $u_h$, that is
\eqn{qquu3}
$$
k_{m+1}^{(h)}\leq c\left(A_0^{(h)} + k_0^{(h)} +  \sum_{i=0}^{m-1} \left[\frac{|\mu_h|(B_i)}{R_{i}^{n-1}}\right]^{\frac{1}{p-1}}\right)  + c\sum_{i=0}^{m-1}  [L_1\omega(R_{i})]^{\frac{2}{p}}(k_i^{(h)}+s)\,,
$$
where here we have obviously set $$k_{m+1}^{(h)}:=|(Du_h)_{B_{m+1}}|\qquad\mbox{and}\qquad
A_0^{(h)}:= \mean{B_{R}} |Du_h- (Du_h)_{B_R}|\, dx $$
 compare with \rif{defiBK} and \rif{deffia}, respectively.
We note that the application of estimate \rif{qquu2} to $u_h$ is legal since the only thing we used in the proof of \rif{qquu2} was that $u \in W^{1,p}_{\loc}(\Omega)$; anyway in the present setting we have $Du \in C^0(\Omega)$ by standard regularity theory. Letting $h \to \infty$ in \rif{qquu3}, and using \rif{convergenza} and the weak convergence of measures, yields
\eqn{nowon}
$$
k_{m+1}\leq c\left(A_0 + k_0 +  \sum_{i=0}^{m-1} \left[\frac{|\mu|(\overline{B_i})}{R_{i}^{n-1}}\right]^{\frac{1}{p-1}}\right)  + c\sum_{i=0}^{m-1}  [L_1\omega(R_{i})]^{\frac{2}{p}}(k_i+s)\,,
$$
where we are adopting the notation established in the proof of Theorem \ref{mainx}. From \rif{nowon} on the rest of the proof follows as for Theorem \ref{mainx} after \rif{qquu2}, but taking into account essentially two facts: first, the convergence in \rif{lebe} just takes place at Lebesgue points, and therefore almost everywhere; second a slight - but obvious - adjustment (enlargement of the balls) has to be made in order to overcome the presence of $|\mu|(\overline{B_i})$ rather than $|\mu|(B_i)$, and to recover the definition of Wolff potential. We have therefore proved the following:
\begin{theorem}\label{mainx2} Let $u \in W^{1,p-1}_0(\Omega)$ be a \textnormal{SOLA} to the problem \trif{Dir1} - which is unique in the case $\mu \in L^1(\Omega)$ -
under the assumptions \trif{asp} and \trif{intdini}. Then there exists a non-negative constant $c \equiv c (n,p,\ratio)$, and a positive radius $\tilde{R} \equiv \tilde{R}(n,p,\ratio,L_1, \omega(\cdot))$ such that
estimate
\trif{mainestx}
holds whenever $B(x_0,2R)\subseteq \Omega$ and $R\leq \tilde{R}$, for almost every $x_0 \in \Omega$. Moreover, when the vector field $a(\cdot)$ is independent of $x$, estimate \trif{mainestx} holds without any restriction on $R$.
\end{theorem}
\begin{remark}\label{uniche} As proved in \cite{boccardo, TW2}, in the case $p=2$ the Solutions Obtained by Limit of Approximations (SOLA) of previous theorem are unique in the sense that if $v\in W^{1,p-1}_0(\Omega)$ is a distributional solution to \rif{Dir1}$_1$ obtainable as a pointwise limit of solutions $v_h\in W^{1,p-1}_0(\Omega)$ to problems of the type \rif{Dirapp} with $\mu_h$ replaced by $\tilde{\mu}_h$ and $\tilde{\mu}_h \rightharpoonup \mu$, then we have $u\equiv v$. Moreover, as remarked a few lines above, the same uniqueness result holds when $\mu \in L^1(\Omega)$.
\end{remark}
In the case $\mu \in W^{-1,p'}(\Omega)$, we can deal with standard weak solutions, i.e. $u \in W^{1,p}_{\loc}(\Omega)$, by mean of local approximations of the type
\eqn{innapp}
$$
\left\{
    \begin{array}{cc}
    -\divo \ a(x,Du_h)=\mu_h & \qquad \mbox{in $\Omega'$}\\
        u_h= u&\qquad \mbox{on $\partial\Omega'$}
\end{array}\right.
$$
whenever $\Omega' \Subset \Omega$ are smooth sub-domains, just along the lines of the proof above. In this case the uniqueness of solutions obtainable in the limit is obviously guaranteed as all solutions are of class $W^{1,p}(\Omega')$. We recall that a characterization of those measures $\mu$ such that $\mu \in W^{-1,p'}(\Omega)$ is that
\eqn{crime}
$$
\int_\Omega {\bf W}^{\mu}_{1, p}(x,1)\, d|\mu|(x)=  \int_\Omega \int_0^1 \left(\frac{|\mu|(B(x,\varrho))}{\varrho^{n- p}}\right)^{\frac{1}{p-1}}\, \frac{d\varrho}{\varrho}\, d|\mu|(x)< \infty
$$
as proved in \cite{HW}; see also \cite[Theorem 4.7.5]{Z}. In particular a measure $\mu \in \MM(\Omega)$ satisfying the density condition $$|\mu|(B_R)\lesssim R^{n-p+\ep}$$ for some $\ep > 0$ satisfies \rif{crime}; for this case see also \cite{K1, K2, liebe}. We again recall that here we have trivially extended $\mu$ to the whole $\er^n$. Moreover the approximation scheme in \rif{innapp} is not necessary when $\mu$ is an integrable function. We recall that the condition $\mu \in L^{\frac{np}{np-n+p}}(\Omega)$ for $p\leq n$ implies $\mu \in W^{-1,p'}(\Omega)$ simply by Sobolev embedding theorem. In this case the proof of Theorem \ref{mainx} in Section \ref{ellipticse} works directly and leads to \rif{mainestx} which thereby holds almost everywhere. Summarizing
\begin{theorem}\label{mainx22} Let $u \in W^{1,p}_{\loc}(\Omega)$ be a local weak solution to \trif{baseq},
under the assumptions \trif{asp} and \trif{intdini}, and with $\mu \in W^{-1, p'}(\Omega)$. Then there exists a constant $c \equiv c (n,p,\ratio)>0$, and a positive radius $\tilde{R} \equiv \tilde{R}(n,p,\ratio,L_1, \omega(\cdot))$ such that
estimate
\trif{mainestx}
holds whenever $B(x_0,2R)\subseteq \Omega$ and $R\leq \tilde{R}$, for almost every $x_0 \in \Omega$. Moreover, when the vector field $a(\cdot)$ is independent of $x$, estimate \trif{mainestx} holds without any restriction on $R$.
\end{theorem}
\begin{remark} By recent higher differentiability result for solutions to measure data problems - see \cite[Theorem 1.5]{mis2} - in the last two theorems estimate \trif{mainestx} does not only hold for a.e.~$x_0 \in\Omega$, but indeed outside a set with Hausdorff measure less than $n-1$. In the case $\mu \in L^\gamma_{\loc}(\Omega)$, with $1<\gamma \leq np/(np-n+p)$, the Hausdorff dimension decreases up to $n-\gamma$.
\end{remark}
\subsection{General parabolic problems}\label{appar} Here, along the lines of the previous section, we give the extension of estimates \rif{parest1} to parabolic SOLA of \rif{parpro} for general measures $\mu \in \MM(\Omega_T)$, following this time the approximation schemes in \cite{Bmany, Dal}. To this aim we shall consider Cauchy-Dirichlet problems of the type
\begin{equation}\label{apprpar}
    \left\{
    \begin{array}{cc}
    \partial_t u_h-\divo \, a(x,t,Du_h)=\mu_h&\mbox{in $ \Omega_T$}\\[3pt]
    u_h=0&\mbox{on $\partial_{\rm par}\Omega_T$\,,}
    \end{array}
    \right.
\end{equation}
solved by a functions $u_h \in C^0(-T, 0;L^2(\Omega))\cap L^2(-T, 0;W^{1,2}_0(\Omega))$, $h \in \en$. As in the elliptic case $\mu_h$ is a sequence of smooth functions obtained from $\mu$ via mollification, and therefore such that $\mu_h \rightharpoonup \mu$ weakly in the sense of measures. The main convergence result in \cite{Bmany} claims that, up to extracting a not-relabeled subsequence, we have that there exists a function $u \in L^r(-T, 0; W^{1,q}(\Omega))$ such that $u_h \rightharpoonup u$ weakly in $L^r(-T, 0; W^{1,q}(\Omega))$, for every choice of the parameters
\eqn{choicepar}
$$q < \frac{n}{n-1}\,, \qquad \qquad  r\in [1,2]\,,\qquad  \qquad \frac{2}{r}+\frac{n}{q}>n+1\,,$$
see the assumptions of \cite[Theorem 1.2]{Bmany}, while the result is proved in \cite[Section 3]{Bmany}. In particular we can choose $q=r>1$ still matching \rif{choicepar} and therefore the main convergence result in \cite[Theorem 3.3]{Bmany} states that, always up to a not-relabeled subsequence, $Du_h \to Du$ a.e. in $\Omega_T$. Then, since the sequence $\{Du_h\}$ is still bounded in $L^q(\Omega_T)$ we infer that $Du_h \to Du$ strongly in $L^1(\Omega_T)$, and finally $u$ weakly solves \rif{parpro} in the sense of \trif{weakp}; in other words $u$ is a SOLA to \rif{parpro}. We recall that uniqueness of SOLA still holds in the parabolic case provided $\mu \in L^1(\Omega_T)$; for this we refer to \cite{Dal}. At this stage, we can repeat the strategy of the elliptic proof of Theorem \ref{mainx2}: i.e. we go to the proof of Theorem \ref{mainp}, we apply it to every solutions $u_h$ in order to have inequalities as in \rif{qquu3}, we pass to the limit with respect to $h \in \en$ the resulting estimates on the numbers $k_m^{(h)}\equiv k_m$ - compare with \rif{qquu3} - and then we conclude with the desired pointwise inequality. Summarizing
\begin{theorem}\label{mainxpar2} Let $u \in L^1(-T, 0; W^{1,1}(\Omega))$ be a \textnormal{SOLA} to the problem \trif{parpro} - which is unique in the case $\mu \in L^1(\Omega_T)$ -
under the assumptions \trif{par1} and \trif{intdini}. Then there exists a constant $c\equiv c(n,\nu, L) $ and a radius $\tilde{R}\equiv \tilde{R}(n,\nu, L, L_1, \omega(\cdot)) $ such that estimate
\trif{parest1}
holds whenever $Q(x_0,t_0; 2R)\subseteq \Omega$, for almost every $(x_0,t_0) \in \Omega_T$, provided $R\leq \tilde{R}$. When the vector field $a(\cdot)$ is independent of the space variable $x$, estimate \trif{parest1} holds without any restriction on $R$.
\end{theorem}
We similarly have
\begin{theorem}\label{mainxpar22} Let $u \in L^1(-T, 0; W^{1,1}(\Omega))$ be a \textnormal{SOLA} to the problem \trif{parpro} - which is unique in the case $\mu \in L^1(\Omega_T)$ - under the assumptions \trif{par2}. Then there exists a constant $c\equiv c(n,\nu, L) $ and a radius $\tilde{R}\equiv \tilde{R}(n,\nu, L, L_1, \omega(\cdot)) $ such that estimate
\trif{parest1}
holds whenever $Q(x_0,t_0; 2R)\subseteq \Omega$, for almost every $(x_0,t_0) \in \Omega_T$, provided $R\leq \tilde{R}$. When the vector field $a(\cdot)$ is independent of the space variable $x$, estimate \trif{parest1} holds without any restriction on $R$.
\end{theorem}
Moreover, when $\mu$ is regular enough - for instance when it belongs to the dual space $L^2(-T, 0; W^{-1,2}(\Omega))$ - estimate \rif{parest1} works for general weak solutions $u \in L^2(-T, 0; W^{1,2}(\Omega))$ to \rif{basicpar}.
\section{Integrability estimates}\label{localest}
In this section we rapidly show how to get the integrability estimates directly from pointwise estimates as \rif{mainestx} and \rif{parest1}; the most interesting results are probably those in the parabolic case, since they are completely new, as described below, while in the elliptic case we shall mostly find a unified approach to the proof of several results appearing in the literature. In both the elliptic and the parabolic case we shall confine ourselves to outline the strategy of proof, since then the details can be easily added by the interested reader.

A main ingredient in the elliptic case here is of course the possibility to control the non-linear Wolff type potentials via the non-linear  Havin-Maz'ja potential, and ultimately via Riesz potentials, which is described in \rif{rieszbound}.
Therefore the main property to use is the boundedness of the Riesz potential in Lebesgue spaces
\eqn{riri1}
$$
\|I_{\beta}(g)\|_{L^{\frac{n q}{n-\beta q}}(\er^{n})}\leq c \|g\|_{L^q(\er^{n})} \qquad \qquad  q \in (1,n/\beta)\,,
$$
with similar estimates to hold also in Marcinkiewicz and Orlicz spaces via interpolation; 
see also \cite{mis2} and next section. Using \rif{riri1} and \rif{rieszbound} we then come to various mapping properties as for instance
\eqn{wolff}
$$
 \mu \in L^q  \Longrightarrow \npma(\cdot, R) \in L^{\frac{nq(p-1)}{n-q}} \qquad  \qquad  q \in (1,n)\,.
$$
The last implication comes of course with related quantitative estimates. In turn, using \rif{wolff} in combination with \rif{mainestx} leads to deduce the following inclusion for solutions to \rif{Dir1} considered in Theorems \ref{mainx2}-\ref{mainx22}:
\eqn{besti}
$$
 \mu \in L^q  \Longrightarrow Du \in L^{\frac{nq(p-1)}{n-q}}_{\loc}(\Omega) \qquad   \qquad q \in (1,n)\,,
$$
The last implication allows to recover various estimates available in the literature \cite{BG1, BG2, DM, I} which are usually achieved via very different techniques, according to the size of $q$. There is now a large number of possible variants to the last result. In fact by the known bounds on the Riesz potential, it is possible to easily find virtually {\em all kinds of rearrangement invariant function spaces integral estimates}. For more details on this we refer to the next section, which is dedicated to parabolic problems, and to Remark \ref{borderliner} below. Further developments of our viewpoint can be found in the forthcoming paper of Cianchi \cite{Ci2}, where a few interesting consequences of estimate \rif{mainestx} are presented. For instance, in \cite{Ci2} the following consequence of \rif{minass} is presented
\eqn{minass2}
$$\mu \in L\left(n,\frac{1}{p-1}\right) \Longrightarrow Du \in L^{\infty}\qquad \mbox{locally in}\ \Omega\,.$$
In this case something more can be actually asserted, namely the continuity of $Du$, as will be proved in the subsequent paper \cite{DMcont}. Again, Wolff potentials based continuity criteria for the gradient of solutions can be obtained starting from the methods presented here, and will be presented in \cite{DMcont}.

We also mention that by {\em localizing} inclusions \rif{wolff} and \rif{besti} it is also possible to obtain explicit local estimates related to the regularity results mentioned above, something which is also not easily reachable via the global techniques usually used in the literature. For this we also refer to \cite{mis, mis2, mis2}.
\begin{remark}[Borderline cases]\label{borderliner} In the regularity theory of measure data problems, for a long while it has been an open issue to establish the validity of estimates in Marcinkiewicz spaces in the conformal case $p=n$. Specifically, for a SOLA solution $u$ of \rif{baseq} we have the inclusion $Du \in \MM^n(\Omega)$ - see \rif{defilo2} below for the definitions. This has been first proved in \cite{DHM} directly for the $p$-Laplacean system, see also \cite{mis} for different approaches. Here we would like to outline how estimates \rif{mainestx}-\rif{mainestxr} immediately imply that $Du \in \MM_{\loc}^n(\Omega, \er^n)$ in our setting. This goes via analyzing properties of Wolff potentials in Lorentz and Marcinkiewicz spaces - see \rif{defilo1}-\rif{defilo2} below for the definitions; indeed by standard properties of Riesz potentials - see for instance \cite[Section 4]{mis2} - and \rif{rieszbound} we have
$$
 \mu \in \MM(\Omega)  \Longrightarrow \npma(\cdot, R) \in \MM^{\frac{n(p-1)}{n-1}}\,. $$
More in general, as in \rif{wolff} we have that \rif{rieszbound} implies
\eqn{wolfflo2}
$$
 \mu \in L(q,\gamma)  \Longrightarrow \npma(\cdot, R) \in L\left(\frac{nq(p-1)}{n-q},\gamma(p-1)\right) $$ whenever $q \in (1,n)$ and $\gamma \in (0,\infty]$.
Another open regularity problem was the borderline $L^p$-regularity of solutions of SOLA to \rif{plap}, this means establishing, in the case $p<n$,  that
\eqn{implylo}
$$
\mu \in L\left(\frac{np}{np-n+p}, \gamma\right) \Longrightarrow Du \in L(p,\gamma(p-1))
$$
whenever $\gamma \in (0,\infty]$. This problem, raised several times in the literature - see for instance \cite{boccann, KiL} and related references - has been recently settled in \cite[Theorem 2]{mis2}. Here another proof of \rif{implylo} follows directly from estimates \rif{mainestx}-\rif{mainestxr} via \rif{wolfflo2}.
\end{remark}
\begin{remark}[Sharpness]\label{wsharp} Inequality \rif{mainestx} is in a sense optimal as $\npma$ allows to recast all the integral estimates in the scale of Lebesgue and Lorentz spaces as described in the previous remark. This is obvious in the case $p=2$ due to well-known representation formulas for the Poisson equation \rif{stima0}; the optimality for $p>2$ follows since the result in \rif{besti} is the best possible; at this point a better potential in \rif{mainestx} would imply a better estimate in \rif{besti}, which is indeed impossible. In other words, in \rif{mainestx} no potential of the type ${\bf W}_{\beta, p}^\mu$ with $\beta>1/p$ can replace $\npma$, for every $p \geq 2$. Let us now test estimate \rif{mainestx} in the case we have the following Dirichlet problem for $p \leq n$:
\eqn{ddiirr}
$$
\left\{
\begin{array}{cccc}
-\triangle_p u & = & \delta_n &\mbox{in}\ B_1\\[5 pt]
u&=&0&\mbox{on}\ \partial B_1
\end{array}\right.
$$
with $\delta_n := \sigma_{n-1}\delta$,
where $\delta $ is the Dirac measure charging the origin and $\sigma_{n-1}:= \H^{n-1}(\partial B_1)$ is the $(n-1)$-dimensional Hausdorff measure of $\partial B_1$. The unique solution to \rif{ddiirr} is given by the so called fundamental solution of the $p$-Laplacean operator, that is
$$
u(x):=\left\{
\begin{array}{cccc}
\frac{p-1}{n-p}~\left(|x|^{\frac{p-n}{p-1}}-1\right)  &\mbox{if}& p< n\\[5 pt]
-\log |x|&\mbox{if} & p=n\,.
\end{array}\right.
$$
The uniqueness is a consequence of classical results of Serrin \cite{serrin} and Kichenassamy \& Veron \cite{KV}; an account of the proof can be found in \cite[Section 4.4]{minq}. Let us show that in this case estimate \rif{mainestx} reverses: let us take $x_0\not= 0$ such that $|x_0|\leq 1/2$, then we have
\eqn{invstima}
$$
\mean{B(x_0,R)}|Du|\, dx  +{\bf W}_{\frac{1}{p},p}^{\delta_n}(x_0,2R) \leq c|Du(x_0)|\,,
$$
with $R=|x_0|$, and for a constant $c$ depending only on $n$ and $p$.
Indeed by co-area formula we have
\begin{eqnarray*}
\mean{B(x_0,R)}|Du|\, dx& \leq  &c(n)\mean{B(0,2R)}|Du|\, dx \\ & =& \frac{c(n)}{R^n}\int_0^{2R} \int_{\partial B_\varrho} |Du(y)|\, d\H^{n-1}(y)\, d \varrho \\ &= &  \frac{c(n,p)}{R^n}\int_0^{2R} \varrho^{n+\frac{1-n}{p-1}}\, \frac{d \varrho}{\varrho}\\
&=& c(n,p)R^{\frac{1-n}{p-1}} \leq c(n,p) |Du(x_0)|\,.
\end{eqnarray*}
On the other hand, as $\delta_n(B(x_0,\varrho))=0$ for $\varrho\leq R$ and $\delta_n(B(x_0,\varrho))=1$ otherwise, we also have
$$
{\bf W}_{\frac{1}{p},p}^{\delta_n}(x_0,2R) \leq c (n,p) \int_{R}^{2R}\varrho^{\frac{1-n}{p-1}} \, \frac{d \varrho}{\varrho}\leq c(n,p)R^{\frac{1-n}{p-1}} \leq c(n,p) |Du(x_0)|\,,
$$
so that \rif{invstima} follows combining the last two inequalities. Inequalities as \rif{invstima} do not only hold when the right hand side of the equation is a Dirac measure; indeed, when considering more general measures concentrating on lower dimensional set -- lines, hyper-planes, etc -- we observe power-type singularities of solutions and estimates like \rif{invstima}.
\end{remark}
\subsection{Parabolic integral estimates}\label{parri} In the parabolic case the consequences of the pointwise estimates are more interesting as we are able to get sharp borderline estimates in Lorentz and Marcinkiewicz spaces which do not seem immediately reachable via the known parabolic techniques. Moreover, as in the elliptic case, we catch directly borderline cases in Lorentz spaces. To this aim let us recall that a map $g : A \subseteq \er^{n+1} \to R^k$, with $A$ being an open subset and $k \in \en$, and for $q \in [1,\infty)$ and $\gamma \in (0,\infty]$, is said to belong to the Lorentz space
$L(q,\gamma)(A)$ iff:
\eqn{defilo1}
$$\|g\|_{L(q,\gamma)(A)}^\gamma:= \gamma\int_0^{\infty}
\left(\lambda^q|\{x \in A \ : \ |g(x)|> \lambda \}|\right)^{\frac{\gamma}{q}} \, \frac{d\lambda}{\lambda}< \infty \quad \mbox{for}\ \gamma< \infty\,.$$
In the case $\gamma=\infty$ we have instead Marcinkiewicz spaces $\MM^q(A)\equiv L(q,\infty)(A)$:
\eqn{defilo2}
$$
\sup_{\lambda > 0} \, \lambda^q|\{x \in A \ : \ |g(x)|> \lambda\}|=:
\|g\|_{\MM^{q}(A)}^q< \infty\;.
$$
The local variants of such spaces are then defined in the obvious way. We here just recall that Lorentz spaces \ap tune" Lebesgue spaces via the second index according to the following chain of inclusions:
$$
L^r \equiv L(r,r)\subset L(q,\gamma) \subset L(q,q) \subset
L(q,r)\subset L(\gamma, \gamma)\equiv L^\gamma\,,$$
which hold whenever $0 < \gamma < q  < r \leq \infty$.
The following inequalities can now be obtained using the definition of caloric Riesz potentials given in \rif{parr}:
\eqn{riri22}
$$
\|I_{\beta}(g)\|_{L\left(\frac{N q}{N-\beta q},\gamma\right)(\er^{n+1})}\leq c \|g\|_{L(q, \gamma)(\er^{n+1})}
$$
which holds whenever $q>1$ and $\beta q< N$, and, in the case $q=1$,
\eqn{riri2}
$$
\|I_{\beta}(g)\|_{\MM^{\frac{N }{N-\beta }}(\er^{n+1})}\leq c|g|(\er^{n+1})\,.
$$
The last inequality still holds when $g$ is a measure and $\beta < N$. Inequalities \rif{riri22} and \rif{riri2} can be localized, more precisely it is not difficult to see that $I_{\beta}(\cdot)$ controls the localized potential ${\bf I}_{1}^\mu(\cdot;R)$ defined in \rif{calpot} in the sense that
$
{\bf I}_{\beta}^\mu(\cdot;R)\lesssim I_{\beta}(g),
$
holds. Therefore we have that
\eqn{imap}
$${\bf I}_{\beta}^\mu(\cdot;R)\colon  L(q,\gamma) \to L\left(\frac{N q}{N-\beta q},\gamma\right) \qquad \mbox{and}\qquad {\bf I}_{\beta}^\mu(\cdot;R)\colon L^1 \to \MM^{\frac{N }{N-\beta }}\,$$ hold, continuously, assuming $\beta q < N$ and $\beta < N$, respectively.
  The estimates in \rif{imap} refer to $\er^{n+1}$ as supporting domain, but they can be localized using cut-off functions and therefore yield local estimates for caloric Riesz potentials; see for instance \cite{mis2, mis3}. The final outcome is the following result:
\begin{theorem}\label{intep1} Let $u \in L^1(-T, 0; W^{1,1}(\Omega))$ be a \textnormal{SOLA} to the problem \trif{parpro} - which is unique in the case $\mu \in L^1(\Omega_T)$ - under the assumptions \trif{par1} and \trif{intdini}. Then the following implications hold:
\eqn{intep2}
$$
\mu \in L(q, \gamma) \Longrightarrow Du \in L\left(\frac{N q}{N-q},\gamma\right) \qquad \mbox{for}\ 1<q< N \ \mbox{and}\  0 <\gamma< \infty
$$
and
\eqn{intep3}
$$
\mu \in \MM(\Omega) \Longrightarrow Du \in \MM^{\frac{N }{N-1}}
$$
with all the previous inclusions being meant as {\bf local} in $\Omega$. \end{theorem}
The previous result is of course a straightforward consequence of estimate \rif{parest1} together with the properties in \rif{imap}. Moreover, using a localized version of the estimates relative to \rif{imap} it is also possible to obtain {\em explcit local estimates} relative to the implications in \rif{intep2}-\rif{intep3}. See for instance the strategies adopted in \cite{mis2}.

\end{document}